\documentclass[letterpaper,11pt]{amsart}
\usepackage{amssymb}
\usepackage{amsthm}
\usepackage[utf8]{inputenc}
\usepackage[margin=1.0in]{geometry}
\usepackage{pst-node}
\usepackage{tikz-cd} 
\usepackage{hyperref}
\newtheorem{theorem}{Theorem}[section]
\newtheorem{lemma}[theorem]{Lemma}
\newtheorem{proposition}[theorem]{Proposition}
\newtheorem{corollary}[theorem]{Corollary}

\newtheorem{condition/definition}[theorem]{Condition/Definition}

\usepackage[T1]{fontenc}
\usepackage{appendix}
\usepackage[english]{babel}
\usepackage[utf8]{inputenc}
\usepackage[backend=biber,style=alphabetic]{biblatex}
\theoremstyle{definition}
\newtheorem{definition}[theorem]{Definition}
\theoremstyle{remark}
\newtheorem{remark}[theorem]{Remark}
\theoremstyle{definition}
\newtheorem{Example}[theorem]{Example}
\theoremstyle{definition}

\usepackage{amsmath}
\usepackage{amsfonts}
\usepackage{tikz-cd}

\def\ol{\overline}

\def\deg{\mathrm{deg}}

\newcommand{\Qp}{\mathbb Q_p}
\def\Proj{\mathrm{Proj}}
\def\Frac{\mathrm{Frac}}

\def\mf{\mathfrak}
\def\ra{\rightarrow}

\def\lim{\mathop{\rm lim}\nolimits}

\def\Spec{\mathop{\rm Spec}}
\def\Spa{\mathop{\rm Spa}}
\def\Spd{\mathop{\rm Spd}}

\def\Ker{\text{Ker}}

\def\Coker{\text{Coker}}

\def\Bun{\mathrm{Bun}}

\def\Gr{\mathrm{Gr}}

\def\Perf{\mathrm{Perf}}

\def\Ad{\mathrm{Ad}}
\def\Lie{\mathrm{Lie}}
\def\dim{\mathrm{dim}}

\def\Cofiber{\mathrm{Cofiber}}

\def\D{\mathrm{D}}

\def\RHom{R\mathcal{H}om}

\def\GL{\mathrm{GL}}

\def\free{\mathrm{free}}
\def\tors{\mathrm{tors}}
\def\bb{\mathbb}

\def\Gal{\mathrm{Gal}}

\newcommand{\mc}{\mathcal}

\title{A Jacobian Criterion for Artin $v$-stacks}
\author{Linus Hamann}
\usepackage[english]{babel}
\usepackage[utf8]{inputenc}
\begin{document}
\maketitle
\textbf{Abstract.} We prove a generalization of the Jacobian criterion of Fargues-Scholze for spaces of sections of a scheme smooth quasi-projective over the algebraic Fargues-Fontaine curve \cite[Section~IV.4]{FS}. Namely, we show how to use their criterion to deduce an analogue for spaces of sections of a smooth Artin stack over the algebraic curve obtained by taking the stack quotient of such a relatively smooth quasi-projective scheme by the action of a linear algebraic group. As an application, we show various moduli stacks appearing in the Fargues-Scholze geometric Langlands program are cohomologically smooth Artin $v$-stacks and compute their $\ell$-dimensions. 
\tableofcontents
\section*{Acknowledgements} I would like to thank Dennis Gaitsgory, David Hansen, Sean Howe, and Peter Scholze for some helpful discussions related to this work. I would also like to thank the anonymous referee, for making many very helpful suggestions. There was a mistake in the proof of \S 4.1 in the published version found by A.I  and has been corrected in the present version of the article. The proof of Lemma \ref{lemma: NoetherianApproximation} was supplied by the A.I.
\section*{Notation and Conventions}
\begin{enumerate}
    \item Let $\ell \neq p$ be distinct prime numbers.
    \item Let $\mathbb{Q}_{p}$ denote the $p$-adic numbers. 
    \item Let $\Perf$ denote the category of affinoid perfectoid spaces in characteristic $p$.  For $S \in \Perf$, we let $\Perf_{S}$ denote the category of affinoid perfectoid spaces over $S$.
    \item For $S \in \Perf$, we will write $X_{S}$ for the algebraic relative Fargues-Fontaine curve over $S$, as defined in the discussion following \cite[Definition~3.3.2]{CS1}. 
    \item We will always use $F$ to denote an algebraically closed complete non-archimedean field in characteristic $p$. 
    \item Given a scheme $S$, an Artin $S$-stack $X$ is a stack $X \rightarrow S$ admitting a smooth representable surjective morphism from a scheme $U$ locally of finite presentation over $S$ such that the diagonal map $X \rightarrow X \times_{S} X$ is representable, separated, and quasi-compact (a variant of \cite[Definition~4.1]{LMB} where we insist $U$ is finitely presented).
    \item We will freely use the formalism of \'etale cohomology of diamonds and $v$-stacks, as developed in \cite{Ecod,FS}. In particular, we will make regular use of the notion of Artin $v$-stacks. We recall \cite[Section~IV.1]{FS} that this is a small $v$-stack $X$ such that the diagonal map $X \rightarrow X \times X$ is representable in locally spatial diamonds, and such that there is some surjective map $f: U \rightarrow X$ from a locally spatial diamond such that $f$ is separated and cohomologically smooth. 
    \item We recall that a cohomologically smooth map $f: X \rightarrow Y$ of Artin $v$-stacks is of $\ell$-dimension $d \in \frac{1}{2}\mathbb{Z}$ if $Rf^{!}(\mathbb{F}_{\ell})$ sits locally in homological degree $2d$ \cite[Definition~IV.1.17]{FS}. Unfortunately, as it stands, this is the only well-behaved notion of dimension for diamonds and $v$-stacks  (see the discussion surrounding \cite[Problem~I.11.1]{FS}); however, for most of the key applications to $\ell$-adic sheaves it is sufficient. We note that, if $f: X \rightarrow Y$ is cohomologically smooth, the sheaf $Rf^{!}(\mathbb{F}_{\ell})$ will be locally constant. In particular, any such map decomposes into a disjoint union of maps $f_{d}: X_{d} \rightarrow Y$ that are pure of $\ell$-dimension $d$. If an Artin $v$-stack $X$ is cohomologically smooth of pure $\ell$-dimension over some specified base then we will write $\dim_{\ell}(X)$ for the $\ell$-dimension. 
\end{enumerate}
\section{Introduction}
For $G/\mathbb{Q}_{p}$ a connected reductive group, Fargues and Scholze \cite{FS} have recently been able to develop the geometric tools necessary to make sense of objects like the moduli stack of $G$-bundles on the Fargues-Fontaine curve and have used it to construct a general candidate for the local Langlands correspondence of $G$. One of the key new geometric tools that aids their analysis is a kind of Jacobian criterion for certain locally spatial diamonds. In particular, if one has a scheme $Z$ smooth quasi-projective over the algebraic Fargues-Fontaine curve $X_{S}$ one can consider the following moduli space attached to it. 
\begin{definition}
For $Z$ a smooth quasi-projective scheme over $X_{S}$, we define $\mathcal{M}_{Z} \rightarrow S$ to be the functor on $\Perf_{S}$, which parametrizes, for $T \in \Perf_{S}$, sections
\[
\begin{tikzcd}
&  & Z \arrow[d] \\
& X_{T} \arrow[ur,"s",dotted] \arrow[r] & X_{S}.
\end{tikzcd}
\]
\end{definition}
We recall classically (i.e. the situation where $X_{S}$ is the base-change of a smooth projective curve over an algebraically closed field $F$ and the test objects are $F$-schemes) that, if we have an $F$-point of $\mathcal{M}_{Z}$, corresponding to a map $X_{F} \rightarrow X_{S}$ and a section $s: X_{F} \rightarrow Z$ over this map, the tangent space at the point defined by $s$ should be given by $H^{0}(X_{F},s^{*}T_{Z/X_{S}})$, and the obstruction space to higher order deformations around this point should be given by $H^{1}(X_{F},s^{*}T_{Z/X_{S}})$, where $T_{Z/X_{S}}$ denotes the relative tangent bundle of $Z \ra X_{S}$. In particular, we should expect smoothness of $\mathcal{M}_{Z}$ around the point defined by $s$ to hold if this obstruction space vanishes, and that, in this case, the dimension locally around the point defined by $s$ is the dimension of the $F$-vector space $H^{0}(X_{F},s^{*}(T_{Z/X_{S}}))$. Over the Fargues-Fontaine curve, making sense of such a statement is a bit subtle. In particular, $H^{0}(X_{F},s^{*}(T_{Z/X_{S}}))$ is far from being finite-dimensional as a vector space over $\mathbb{Q}_{p}$ (the ring of global sections of $X_{F}$). However, one can extract a finite-dimensional space from the global sections by considering the pro-\'etale sheaf $\mathcal{H}^{0}(s^{*}(T_{Z/X_{S}})) \rightarrow \Spa(F)$ on $\Perf_{F}$ which sends $T \in \Perf_{F}$ to the space of sections $H^{0}(X_{T},s^{*}(T_{Z/X_{S}})_{T})$, where $s^{*}(T_{Z/X_{S}})_{T}$ denotes the base change of $s^{*}(T_{Z/X_{S}})$ to $X_{T}$. When $s^{*}(T_{Z/X_{S}})$ has positive slopes, this is representable by a finite-dimensional locally spatial diamond called a Banach-Colmez space, as defined in \cite{LB}, and looks like a perfectoid open unit disc quotiented out by a pro-\'etale group action (See Proposition \ref{prop: repovergeompoint1} (1)). Roughly speaking, we should expect that, infinitesimally around the point defined by $s$, the moduli space $\mathcal{M}_{Z}$ looks like the locally spatial diamond $\mathcal{H}^{0}(s^{*}(T_{Z/X_{S}}))$. This space will define a cohomologically smooth diamond if the pullback $s^{*}(T_{Z/X_{S}})$ has positive Harder-Narasimhan (abbrev. HN)-slopes (if there is some slope $0$ constituent then $\mathcal{H}^{0}(s^{*}(T_{Z/X_{S}}))$ fails to be smooth due to the appearance of profinite sets in the global sections (See Proposition \ref{prop: repovergeompoint1})). Moreover, the positive slope assumption also implies that the obstruction space $H^{1}(X_{F},s^{*}(T_{Z/X_{S}}))$ vanishes. Therefore, we should expect smoothness of the space $\mathcal{M}_{Z}$ around the points defined by such sections. This motivates the following definition.
\begin{definition}{\label{def: Mzsm}}
For $Z$ as above, we let $\mathcal{M}_{Z}^{\mathrm{sm}} \subset \mathcal{M}_{Z}$ denote the open sub-functor parameterizing sections $s: X_{T} \rightarrow Z$ such that the pullback of the tangent bundle $s^{*}(T_{Z/X_{S}})$ has positive HN-slopes after pulling back to any geometric point of $T$.  
\end{definition}
Then a slight reformulation of the Jacobian criterion of Fargues-Scholze using GAGA (\cite[Proposition~II.2.7]{FS},\cite[Theorem~8.7.7]{KL}) is as follows.
\begin{theorem}{\cite[Theorem~IV.4.2]{FS}}{\label{thm: jacobiancriterion}}
For $S \in \Perf$ and $Z \rightarrow X_{S}$ a scheme smooth quasi-projective over $X_{S}$, the $v$-sheaf $\mathcal{M}_{Z}$ defines a locally spatial diamond, the map $\mathcal{M}_{Z} \rightarrow S$ is compactifiable, and the map $\mathcal{M}_{Z}^{\mathrm{sm}} \rightarrow S$ is cohomologically smooth.
\\\\
Moreover, for any geometric point $x: \Spa(F) \rightarrow \mathcal{M}_{Z}^{\mathrm{sm}}$, given by a map $\Spa(F) \rightarrow S$ and a section $s: X_{F} \rightarrow Z$, the map $\mathcal{M}_{Z}^{\mathrm{sm}} \rightarrow S$ has $\ell$-dimension at $x$ equal to the degree of $s^{*}(T_{Z/X_{S}})$.
\end{theorem}
\begin{remark}
Fargues and Scholze formulate their result with the adic Fargues-Fontaine curve and adic spaces over it; however, as we will want to later consider Artin stack quotients of the spaces $Z$, we find it more technically convenient to stick to the schematic formalism. This has the small caveat that one needs to restrict to affinoid perfectoid spaces, as one cannot make sense of gluing affinoids together for the schematic Fargues-Fontaine curve. Restricting to affinoids is a rather minute restriction, since the category of perfectoid spaces with any reasonable topology is generated by affinoids. 
\end{remark}
We briefly recall the importance of this result. In particular, Fargues and Scholze can use this to show that $\Bun_{G}$ is a cohomologically smooth Artin $v$-stack of $\ell$-dimension $0$, and is more importantly equipped with certain explicit smooth charts that have many nice geometric properties (e.g.\cite[Proposition~V.4.2]{FS}). Namely, in \cite[Section~V.3]{FS}, Fargues and Scholze consider moduli spaces of $P$-bundles, where $P \subset G$ is a proper parabolic subgroup. These give rise to cohomologically smooth locally spatial diamonds, represented by an iterated fibration of negative Banach-Colmez spaces (cf. the proof of Proposition \ref{prop: projcohsmooth}). One uses these moduli spaces to uniformize $\Bun_{G}$ by sending a parabolic structure to its induced $G$-bundle. To verify the desired claim, one needs to show that the fibers of this uniformization map are cohomologically smooth and compute their $\ell$-dimension. This is where the Jacobian criterion comes in. In particular, the moduli space of $P$-structures on a fixed $G$-bundle $\mathcal{E}$ on $X_{S}$ is a space of sections of a scheme  $\mathcal{E}^{\text{gm}}/P$ projective over $X_{S}$, where $\mathcal{E}^{\text{gm}}$ denotes the geometric realization of $\mathcal{E}$. These charts and their smoothness properties serve as a key computational tool for establishing many foundational results on the sheaf theory of $\Bun_{G}$. 

While this story is nice, for many applications one wants to consider moduli spaces like $\Bun_{G}$ and be able to efficiently study their smoothness properties without having to rely on explicit charts. Classically, such a formalism already exists. For example, to any Artin stack $X \rightarrow \Spec{k}$ for $k$ an algebraically closed field, one can associate to it a tangent complex $T^{*}_{X/k}$ (Theorem \ref{thm: tangcompl}). This is a complex of quasi-coherent sheaves on $X$ concentrated in degrees $[-1,\infty)$. The cohomology in degree $0$ encodes the tangent vectors at a point and the cohomology in degree $-1$ encodes the infinitesimal automorphisms at a point, while the higher cohomology encodes information about obstructions to higher order deformations. In particular, (roughly speaking) if one has a $k$-point $x: \Spec{k} \rightarrow X$  then the cohomology of the derived pullback $Lx^{*}(T^{*}_{X/k})$ has vanishing cohomology in degrees $\geq 1$, the moduli space $X$ is smooth at $x$ and locally has dimension equal to the Euler characteristic of the complex $Lx^{*}(T^{*}_{X/k})$. More specifically, infinitesimally around $x$ the space $X$ looks like the Picard groupoid
\[ [H^{0}(Lx^{*}(T^{*}_{X/k}))/H^{-1}(Lx^{*}(T^{*}_{X/k}))] \]
associated to the two-term complex $Lx^{*}(T^{*}_{X/k})$ of vector spaces over $k$ (See \S \ref{subsec: ReviewofTangentComplex} for what this means), where here we regard the cohomology groups as an affine space over $k$. If we let $X = \Bun_{G}$ be the moduli space of $G$-bundles over a smooth projective curve $Y/\Spec{k}$ and consider such a point $x$ corresponding to a $G$-torsor $\mathcal{F}_{G}$ over $Y$ then the cohomology of $Lx^{*}(T^{*}_{X/k})$ is described by the cohomology of the object $\mathcal{F}_{G} \times^{G,\Ad} \mathrm{Lie}(G)[1]$ on $Y$ (See Examples \ref{ex: BunGex} and \ref{ex: sections}). We note that, since $Y$ is a curve, the complex has vanishing cohomology in degree $\geq 1$. Hence, the deformations are always unobstructed and infinitesimally it looks like
\[ [H^{1}(Y,\mathcal{F}_{G} \times^{G,Ad} \mathrm{Lie}(G))/H^{0}(Y,\mathcal{F}_{G} \times^{G,Ad} \mathrm{Lie}(G))] \rightarrow \Spec{k} \]
and from this we can see that $\Bun_{G}$ is smooth of dimension $\dim(G)(g - 1)$, where $g$ is the genus of $Y$. 

The main aim of this note is to develop a formalism that would allow similar arguments to work for spaces occurring in the Fargues-Scholze geometric Langlands program. To get at this, let us consider a different way of understanding the previous argument more in line with the type of analysis used in Theorem \ref{thm: jacobiancriterion}. In particular, recall that $\Bun_{G}$ can be written as the moduli space of sections $[Y/G] \rightarrow Y$, where $[Y/G]$ is the classifying stack of $G$ over $Y$. Given a section $s: Y \rightarrow [Y/G]$ corresponding to a $G$-bundle $\mathcal{F}_{G}$ on $Y$, we can compute that the pullback of the tangent complex $T^{*}_{[Y/G]/Y}$ is isomorphic to the complex of vector bundles $\mathcal{F}_{G} \times^{G,Ad} \mathrm{Lie}(G)[1]$, and, as seen above, the cohomology of this complex controls the deformation theory of $\Bun_{G}$ around the point corresponding to $s$, in perfect analogy to Theorem \ref{thm: jacobiancriterion}. 
 
With this in mind, it becomes tempting to consider an analogue of Theorem \ref{thm: jacobiancriterion} for spaces of sections $\mathcal{M}_{Z}$, where $Z$ is an Artin $X_{S}$-stack with an atlas given by a scheme smooth quasi-projective over $X_{S}$. We show that such an analogue indeed exists for $Z$ which are obtained as stack quotients of a scheme smooth quasi-projective over $X_{S}$ by an action of a linear algebraic group. In particular, given a scheme $Z \rightarrow X_{S}$ smooth quasi-projective over $X_{S}$ with an action of a linear algebraic group $H/\mathbb{Q}_{p}$, we consider the stack quotient $[Z/H] \rightarrow X_{S}$ and the $v$-stack of sections $\mathcal{M}_{[Z/H]} \rightarrow S$. Given a section $s: X_{T} \rightarrow [Z/H]$, we can consider the complex $Ls^{*}(T^{*}_{([Z/H])/X_{S}})$ on $X_{T}$. Since $Z$ is smooth, we can realize this as a two-term complex of vector bundles on $X_{T}$ sitting in cohomological degrees $-1$ and $0$, denoted $\mathcal{E}^{*} := \{\mathcal{E}^{-1} \rightarrow \mathcal{E}^{0}\}$, where we identify $Ls^{*}(T^{*}_{([Z/H])/X_{S}})$ with the cofiber of $\mathcal{E}^{-1} \rightarrow \mathcal{E}^{0}$ in the derived category of quasi-coherent sheaves on $X_{T}$, denoted by $|\mathcal{E}^{*}|$. It therefore follows, since $T$ is affinoid, that the cohomology of this complex is concentrated in degrees $[-1,1]$, and the cohomology in degree $1$ will be controlled by the 1st cohomology group of the $0$th cohomology sheaf of $|\mathcal{E}^{*}|$. In analogy with the above, it should follow that, given an $F$-point of $\mathcal{M}_{[Z/H]}$ corresponding to a map $X_{F} \rightarrow X_{S}$ and a section $s: X_{F} \rightarrow [Z/H]$ over it, we should expect that $s$ defines a smooth point if the $0$th cohomology sheaf of $Ls^{*}(T^{*}_{([Z/H])/X_{S}})$ defines an object with vanishing cohomology in degree $1$. This motivates the following.
\begin{definition}{\label{defn: MZHsmooth}}
For $H$ and $Z$ as above, we let $\mathcal{M}^{\mathrm{sm}}_{[Z/H]} \subset \mathcal{M}_{[Z/H]}$ denote the sub $v$-stack (See Remark \ref{rem: MZsmisavsheaf}) parameterizing sections $s: X_{T} \rightarrow [Z/H]$ such that the pullback $Ls^{*}(T^{*}_{([Z/H])/X_{S}})$ satisfies that its $0$th cohomology sheaf $H^{0}(Ls^{*}(T^{*}_{([Z/H])/X_{S}}))$ with respect to the standard $t$-structure on quasi-coherent $\mathcal{O}_{X_{T}}$-modules is subject to the following condition. For each geometric point $x: \Spa(F) \ra T$, the pullback of $H^{0}(Ls^{*}(T^{*}_{([Z/H])/X_{S}}))$ along the induced map $X_{F} \ra X_{T}$ is a coherent sheaf on $X_{F}$ whose torsion-free quotient has only strictly positive slopes, allowing this quotient to be zero. We recall that this is well-defined, since $X_{F}$ is a Dedekind scheme (\cite[Theorem~13.5.3 (1)]{SW}).
\end{definition}
\begin{remark}
In the case that $H$ is trivial, so that we are in the situation of Theorem \ref{thm: jacobiancriterion}, we note that $H^{0}(Ls^{*}(T^{*}_{([Z/H])/X_{S}}))$ is just the pullback $s^{*}T_{Z/X_{S}}$; in particular, Definition \ref{defn: MZHsmooth} generalizes Definition \ref{def: Mzsm} to this stacky context.
\end{remark}
\begin{remark}
It is not a priori clear that the monomorphism $\mathcal{M}_{[Z/H]}^{\mathrm{sm}} \rightarrow \mathcal{M}_{[Z/H]}$ of $v$-stacks is an open immersion (See Remark \ref{rem: MZsmisavsheaf}); however, we will show this in the course of proving Theorem \ref{thm: stackyjacobi}.
\end{remark}
Our main theorem is as follows.  
\begin{theorem}{\label{thm: stackyjacobi}}
For $S \in \Perf$, let $H$ be a linear algebraic group over $\mathbb{Q}_{p}$ and $Z \rightarrow X_{S}$ a scheme smooth quasi-projective over $X_{S}$ with an action of $H$. Then $\mathcal{M}_{[Z/H]} \rightarrow S$ defines an Artin $v$-stack, $\mathcal{M}_{[Z/H]}^{\mathrm{sm}} \subset \mathcal{M}_{[Z/H]}$ is an open substack, and $\mathcal{M}_{[Z/H]}^{\mathrm{sm}} \rightarrow S$ is a cohomologically smooth map of Artin $v$-stacks.
\\\\
Moreover, for any geometric point $x: \Spa(F) \rightarrow \mathcal{M}_{[Z/H]}^{\mathrm{sm}}$ given by a map $\Spa(F) \rightarrow S$ and a section $s: X_{F} \rightarrow [Z/H]$, the map $\mathcal{M}_{[Z/H]}^{\mathrm{sm}} \rightarrow S$ is locally around $x$ of $\ell$-dimension equal to the quantity: 
\[ \chi(Ls^{*}(T^{*}_{([Z/H])/X_{S}})) := \mathrm{deg}(H^{0}(Ls^{*}(T^{*}_{([Z/H])/X_{S}}))) - \mathrm{deg}(H^{-1}(Ls^{*}(T^{*}_{([Z/H])/X_{S}}))), \]
where $H^{i}$ for $i \in \bb{Z}$ denotes the $i$th cohomology sheaf and $\mathrm{deg}(-)$ denotes the degree of a coherent sheaf on $X_{F}$.
\end{theorem}
\begin{remark}
One could probably prove an analogue of Theorem \ref{thm: stackyjacobi} in the adic formalism, but this would require making sense of stack quotients $[Z/H]$, for $H$ a linear algebraic group and $Z$ a sous-perfectoid space. This is possible by using the fact that $H$-torsors over a sous-perfectoid space are sous-perfectoid (See \cite[Theorem~19.5.2]{SW}). Then one would need to develop a theory of tangent complexes for these objects, using \cite[Section~IV.4.1]{FS} as a starting point. 
\end{remark}
\begin{remark}{\label{rem: Picardvgroupoidlocalmodel}}
If we have a section $s: X_{F} \ra [Z/H]$ defining a geometric point in $\mathcal{M}_{[Z/H]}^{\mathrm{sm}}$ then one can attach a certain Artin $v$-stack $\mc{P}(Ls^{*}T^{*}_{([Z/H])/X_{S}}) \ra  \Spa(F)$ called the Picard $v$-groupoid of $Ls^{*}T^{*}_{([Z/H])/X_{S}}$. The Picard $v$-groupoid has isomorphism classes of objects given by $\mathcal{H}^{0}(Ls^{*}T^{*}_{([Z/H])/X_{S}})$ and the automorphisms of the neutral object are given by $\mathcal{H}^{-1}(Ls^{*}T^{*}_{([Z/H])/X_{S}})$, where $\mathcal{H}^{i}(Ls^{*}T^{*}_{([Z/H])/X_{S}})$ is the functor on $\Perf_{F}$ given by the $i$th hypercohomology of $Ls^{*}T^{*}_{([Z/H])/X_{S}}$. This will be a cohomologically smooth Artin $v$-stack of pure $\ell$-dimension equal to $\chi(Ls^{*}(T^{*}_{([Z/H])/X_{S}}))$ (See Corollary \ref{cor: PicardGroupoidSmooth}). We think of these Picard $v$-groupoids as the infinitesimal linear models for the Artin $v$-stack $\mathcal{M}_{[Z/H]}^{\mathrm{sm}}$. This is analogous to how, in the context of Theorem \ref{thm: jacobiancriterion}, the positive Banach-Colmez spaces $\mathcal{H}^{0}(Ls^{*}T_{Z/X_{S}})$ are the infinitesimal linear models for the locally spatial diamond $\mathcal{M}_{Z}^{\mathrm{sm}}$. In certain cases, one can very directly see the link between these Picard $v$-groupoids and the spaces $\mathcal{M}_{[Z/H]}^{\mathrm{sm}}$ (See Example \ref{ex: negbcspace} and the proof of Proposition \ref{prop: projcohsmooth}). This very nicely parallels the classical story relating these Picard stacks and the tangent complex described above, which we review in \S \ref{subsec: ReviewofTangentComplex}.
\end{remark}
The proof of Theorem \ref{thm: stackyjacobi} is by d\'evissage to Theorem \ref{thm: jacobiancriterion}. The key point is that, by choosing a closed embedding $H \hookrightarrow \GL_{n}$ for some sufficiently large $n$, one can assume that $H = \GL_{n}$. Then one has a natural map of $v$-stacks:
\[ \mathcal{M}_{[Z/\GL_{n}]} \rightarrow \mathcal{M}_{[X_{S}/\GL_{n}]} \simeq \Bun_{\GL_{n},S}. \]
We can then find explicit cohomologically smooth charts for the moduli space $\Bun_{\GL_{n},S}$ which can be written as spaces of sections of a certain Grassmannians $\tilde{Z}_{i}$ over $X_{S}$ analogous to the classical theory of Quot schemes (See Proposition \ref{prop:quot}). By pulling back these charts along the map $\mathcal{M}_{[Z/\GL_{n}]} \rightarrow \mathcal{M}_{[X_{S}/\GL_{n}]} \simeq \Bun_{\GL_{n},S}$, we can find charts for $\mathcal{M}_{[Z/\GL_{n}]}$ which are given as open subspaces of spaces of sections of a fiber product $\tilde{Z}_{i} \times_{[X_{S}/\GL_{n}]} [Z/\GL_{n}]$, which will also be representable by a scheme smooth quasi-projective over $X_{S}$, and this can always be arranged after a pro-étale cover of $S$ (using Proposition \ref{lem:v-local-linearization}, plus a linearization argument). The condition on the pullback of the tangent bundle of this space needed to apply Theorem \ref{thm: jacobiancriterion} can be related to the above condition on the tangent complex of $[Z/\GL_{n}] \ra X_{S}$ via looking at the distinguished triangle of tangent complexes given by the projection $\tilde{Z}_{i} \times_{[X_{S}/\GL_{n}]} [Z/\GL_{n}] \rightarrow [Z/\GL_{n}]$. This will in turn give the desired claim. 

As an application of these ideas, we can give a more direct proof that $\Bun_{G}$ is a cohomologically smooth Artin $v$-stack of $\ell$-dimension $0$ if $G$ is quasi-split, as well as show that the moduli stack $\Bun_{P}$ for $P \subset G$ a parabolic is cohomologically smooth. Perhaps more interestingly however, we will verify that, if $G = \GL_{2}$ and $B$ denotes the Borel then Drinfeld's compactification $\ol{\Bun}_{B}$ is $\ell$-cohomologically smooth. This result plays an important role in the theory of geometric Eisenstein series on the Fargues-Fontaine curve and leads to many simplifications in the theory for groups of type $A_{1}$. We will explore this more in future work. 

In \S $2$, we review the theory of the tangent complex for usual Artin stacks. Section $3$ is a review of the theory of Banach-Colmez spaces, where we also prove various properties of "stacky" Banach-Colmez like spaces obtained from it, as described in Remark \ref{rem: Picardvgroupoidlocalmodel}. In \S $4$, we conclude by giving the proof of Theorem \ref{thm: stackyjacobi} and giving some applications to verifying the cohomological smoothness of various moduli stacks appearing in the Fargues-Scholze geometric Langlands correspondence; in particular, we will show smoothness of Drinfeld's compactification in the case of a Borel inside $\GL_{2}$. In \S $5$, we prove a technical Noetherian approximation lemma for Fargues-Fontaine curves over products of points, which is required for the reduction step in \S $4$ from general $H$ to $\GL_{n}$. 
\section{The Tangent Complex of an Artin Stack}
\subsection{Review of the Tangent Complex}{\label{subsec: ReviewofTangentComplex}}
In this section, we will review the theory of the tangent complex of an Artin stack. Most of this material in the context of schemes can be found in the book of Illusie \cite{Il}, and for Artin stacks in the book of Laumon-Moret--Bailly \cite{LMB}. We also recommend the interested reader take a look at the paper of Olsson \cite{Ols}, where several mistakes in the book of \cite{LMB} related to the functoriality of the lisse-\'etale site and the derived pull-back between maps of Artin stacks are remedied. 

Let's start with some motivation. Let $S$ be a scheme (affine for simplicity) and $X \rightarrow S$ an Artin $S$-stack. Let $D = \Spec{\mathbb{Z}[\epsilon]/\epsilon^{2}}$ be the scheme attached to the ring of dual numbers. Write $S[\epsilon]/\epsilon^{2} := S \times_{\Spec{\mathbb{Z}}} D$ for the dual numbers over $S$. Given a point $x: S \rightarrow X$, we can define a tangent vector of $X$ at 
$x$ to be a lift of the map $x$ to a map $S[\epsilon]/\epsilon^{2} \rightarrow X$. Let $T_{X/S,x}$ be the set of tangent vectors. This comes equipped with several kinds of additional structure. For starters, the elements of $T_{X/S,x}$ are the objects of a groupoid, since they can be realized as a subgroupoid of $X(S[\epsilon]/\epsilon^{2})$ given as the fiber over $x$ of the natural map $X(S[\epsilon]/\epsilon^{2}) \rightarrow X(S)$. Moreover, it has an addition given by $S[\epsilon]/\epsilon^{2} \rightarrow S[\epsilon]/\epsilon^{2} \coprod_{S} S[\epsilon]/\epsilon^{2}$ defined by the map of rings sending $a + b\epsilon_{1} + c\epsilon_{2}$ to $a + (b + c)\epsilon$, where $a,b,c \in \mathcal{O}_{S}$. This gives $T_{X/S,x}$ the structure of a Picard category. In other words, a symmetric monoidal category such that tensoring by a fixed object has an inverse (See \cite[Section~14.4]{LMB} for a more precise definition). Moreover, it has a natural scalar multiplication by $\lambda \in \mathcal{O}_{S}$, given by the map sending $a + b\epsilon \mapsto a + \lambda b\epsilon$. This induces an endomorphism $\lambda: T_{X/S,x} \rightarrow T_{X/S,x}$ satisfying natural compatibilities. This gives it the structure of a Picard $S$-stack, as in \cite[Definition~14.4.2]{LMB}. 

Conversely, suppose we are given a two-term complex
$V^{-1}\xrightarrow{d}V^0$ of sheaves of abelian
groups on $S_{\mathrm{\acute et}}$, placed in degrees $-1$ and $0$.
It defines a Picard prestack whose groupoid over an object $U$
has objects $x\in V^0(U)$ and morphisms from $x$ to $y$ given by
sections $a\in V^{-1}(U)$ satisfying $d(a)=y-x$.
Let $\mathcal P(V^{*})$ denote its associated stack for
the \'etale topology.

The sheaf of isomorphism classes of objects of
$\mathcal P(V^{*})$ is canonically identified with
$H^0(V^*)$, while the sheaf of automorphisms of
the neutral object is canonically identified with
$H^{-1}(V^*)$. Quasi-isomorphisms of complexes induce equivalences of the associated Picard stacks. Thus, up to equivalence, this
construction depends only on the corresponding object of
$\D^{[-1,0]}(S_{\mathrm{\acute et}},\mathbb Z)$ (See \cite[Theorem~14.4.5]{LMB}). Let $V^{*}$ be such a complex associated to $T_{X/S,x}$. The above analysis suggests if $X$ is smooth that, infinitesimally around $x$, the stack $X$ should look like the stack quotient $[V^{0})/V^{-1}] \rightarrow S$. More transparently, the complex $V^{*}$ should arise by applying the Grothendieck construction (\cite[Section~14.2.6]{LMB}) to a two term complex of quasi-coherent $\mathcal{O}_{S}$-modules $C^{*}$, as is done in \cite[Example~14.4.10]{LMB} and this will be the form that the tangent complex takes.

This leads us to the expectation that there should exist an object $T_{X/S}^{*}$ in the derived category of quasi-coherent sheaves on $X$, whose pullback to $x$ gives rise to a complex whose Picard groupoid is precisely $T_{X/S,x}$ if $X$ is smooth. In the general case, we should expect its higher cohomology to carry information about obstructions to higher order deformations. 

In preparation for constructing this object, let us briefly recall the definition of quasi-coherent sheaves on an Artin $S$-stack $X$. We have the following key definition of Laumon-Moret--Bailly.
\begin{definition}{\cite[Definition~12.1]{LMB}\footnote{The definition we give is actually a slight variant of \cite[Definition~12.1]{LMB}, but it is easy to see that they give rise to equivalent sheaf theories.}}
Let $X \rightarrow S$ be an Artin $S$-stack. We define the lisse-\'etale site of $X$, denoted $\text{Lis-\'et}(X)$. The objects are $(U,u)$, for $U$ a scheme over $S$ and $u: U \rightarrow X$ is a smooth morphism over $S$. The morphisms between a pair $(U_{1},u_{1})$ and $(U_{2},u_{2})$ are commutative diagrams 
\[
\begin{tikzcd}
U_{1} \arrow[dd,"\phi"] \arrow[dr,"u_{1}"] & & \\
& X &  \\
U_{2} \arrow[ur,"u_{2}"] & & 
\end{tikzcd}
\]
defined by a choice of natural equivalence $\alpha: u_{1} \simeq u_{2} \circ \phi$. Coverings are families of jointly surjective \'etale morphisms $\{U_{i} \rightarrow U\}_{i \in I}$.
\end{definition}
\begin{remark}{\label{rem: quasicoh}}
It is easily verified \cite[Proposition~12.2.1]{LMB} (See also \cite[Remark~3.2]{Ols}) that a sheaf $\mathcal{F}$ on this site is equivalent to a system of sheaves $\mathcal{F}_{U}$ on $(U)_{\text{\'et}}$ and, for every morphism $\{\phi: U \rightarrow V\} \in \text{Lis-\'et}(X)$, a map 
\[ \theta_{\phi}: \phi^{-1}(\mathcal{F}_{V}) \rightarrow \mathcal{F}_{U} \]
satisfying the following properties:
\begin{enumerate}
    \item $\theta_{\phi}$ is an isomorphism if $\phi$ is \'etale. 
    \item For maps $U \xrightarrow{\phi} V \xrightarrow{\psi} W$, we have $\theta_{\phi} \circ \phi^{-1}(\theta_{\psi}) = \theta_{\psi \circ \phi}$.
\end{enumerate}
\end{remark}
This site has a structure sheaf, denoted $\mathcal{O}_{X_{\text{lis-\'et}}}$, sending $U \in \text{Lis-\'et}(X)$ to $\Gamma(U,\mathcal{O}_{U})$. We can therefore consider $\mathcal{O}_{X_{\text{lis-\'et}}}$-modules on $X$. This allows us to define the following.
\begin{definition}{\cite[Definition~13.2.2]{LMB}}
We say an $\mathcal{O}_{X_{\text{lis-\'et}}}$-module $\mathcal{F}$ on $X$ is a quasi-coherent sheaf if, for all $U \in \text{Lis-\'et}(X)$, $\mathcal{F}_{U}$ is a quasi-coherent sheaf on $(U)_{\text{\'et}}$ and, for all morphisms $\phi: U \rightarrow V$, the map $\phi^{*}(\mathcal{F}_{V}) \rightarrow \mathcal{F}_{U}$ induced by $\theta_{\phi}$ is an isomorphism. It follows, by \cite[Lemma~6.2]{Ols}, that this is equivalent to checking this condition for smooth maps $\phi: U \rightarrow V$. Similarly, we say such an $\mathcal{F}$ is a vector bundle on $X$ if all the restrictions $\mathcal{F}_{U}$ are vector bundles on $(U)_{\text{\'et}}$.
\end{definition}
We let $\mathrm{QCoh}(X)$ denote the category of quasi-coherent sheaves on $X$. From here, we can consider $\D_{\mathrm{qcoh}}(X,\mathcal{O}_{X})$, the full derived sub-category of the derived category of $\mathcal{O}_{X_{\text{lis-\'et}}}$-modules with quasi-coherent cohomology, which has the structure of a triangulated category by \cite[Lemma~13.1.3]{LMB}. We then define $\D^{+}_{\mathrm{qcoh}}(X,\mathcal{O}_{X})$ (resp. $\D^{-}_{\mathrm{qcoh}}(X,\mathcal{O}_{X})$) to be the sub-category with bounded below (resp. above) cohomology. Similarly, we write $\D^{+}_{\mathrm{coh}}(X,\mathcal{O}_{X})$ (resp. $\D^{-}_{\mathrm{pcoh}}(X,\mathcal{O}_{X})$) for the analogous sub-categories of pseudo-coherent objects. We note, by \cite[Proposition~13.2.6]{LMB}, we have a natural internal hom functor
\[
R\mathrm{Hom}_{\mathcal O_{\mathcal X}}(-,-):
\D_{\mathrm{pcoh}}^{-,\mathrm{op}}(X,\mathcal O_{X})
\times \D_{\mathrm{qcoh}}^{+}(X,\mathcal O_{X})
\longrightarrow \D^{+}_{\mathrm{qcoh}}( X,\mathcal O_{ X}).
\]
Now, we will turn to defining the derived pullback for a map of Artin $S$-stacks. Here, one needs to be a bit careful that, given a representable morphism $f: X \rightarrow Y$ of algebraic stacks, there is an induced functor
\[ \text{Lis-\'et}(Y) \rightarrow \text{Lis-\'et}(X) \]
\[ U \mapsto U \times_{Y} X \]
which induces a pair of adjoint functors $(f^{-1},f_{*})$ on sheaves, but this does not give rise to a morphism of topoi because $f^{-1}$ is not left exact (See \cite[Example~3.4]{Ols}).
Nonetheless, we can still define a left adjoint $Lf^{*}$ to $Rf_{*}$, the pushforward in the derived category of quasi-coherent sheaves, at least after taking truncations (See \cite[Section~7]{Ols}). This leads us to consider $\D'_{\mathrm{qcoh}}(X,\mathcal{O}_{X})$ the left-completion of $\D_{\mathrm{qcoh}}(X,\mathcal{O}_{X})$. Namely, the category of systems 
\[ K = \{\cdots \rightarrow K_{\geq -n - 1} \rightarrow K_{\geq - n} \rightarrow \cdots \rightarrow K_{\geq 0} \} \]
where $K_{\geq -n} \in \D^{+}_{\mathrm{qcoh}}(X,\mathcal{O}_{X})$ and the maps
\[ K_{\geq -n} \rightarrow \tau_{\geq -n}K_{\geq -n}  \] 
\[  \tau_{\geq -n}K_{\geq -n - 1} \rightarrow \tau_{\geq -n}K_{\geq -n}  \]
are isomorphisms, where $\tau_{\geq -n}$ denotes the usual truncation functors in degree $\geq -n$. This comes equipped with a notion of distinguished triangle, shift maps, and cohomology functors, as defined in \cite[Section~7]{Ols}. With this in hand, we can define the following. 
\begin{proposition}{\cite[Section~7]{Ols}}{\label{lem: fulltancomp}}
Let $f: X \rightarrow Y$ be a quasi-compact and quasi-separated map of Artin stacks. Then there exists a derived pullback functor 
\[ Lf^{*}: \D'_{\mathrm{qcoh}}(Y,\mathcal{O}_{Y}) \rightarrow \D'_{\mathrm{qcoh}}(X,\mathcal{O}_{X}) \]
obtained by taking the system $\{\tau_{\geq a}Lf^{*}\}$, where $\tau_{\geq a}Lf^{*}$ is the left adjoint to the functor $Rf_{*}: \D^{[a,\infty)}_{\mathrm{qcoh}}(X,\mathcal{O}_{X}) \rightarrow \D^{[a,\infty)}_{\mathrm{qcoh}}(Y,\mathcal{O}_{Y})$ on the derived sub-categories with cohomology concentrated in degrees $[a,\infty)$.
\end{proposition}
With these definitions out of the way, we have our key result on the tangent complex\footnote{We warn the reader that, for our purposes, it will be more convenient to work with the tangent complex, while Olsson and Laumon-Moret--Bailly work with the cotangent complex. However, the claims we will make can be obtained from their work by applying  $R\mathcal{H}om(-,\mathcal{O}_{X})$, as defined above.}.
\begin{theorem}{\cite[Theorem~8.1]{Ols}, \cite[Theorem~17.3]{LMB}}{\label{thm: tangcompl}}
We assume that all Artin stacks are locally Noetherian. Let $f: X \rightarrow Y$ be a finitely presented morphism of Artin stacks. Then to $f$ one can associate an object $T^{*}_{X/Y} \in \D_{\mathrm{qcoh}}^{'}(X,\mathcal{O}_{X})$ called the tangent complex of $f$ satisfying the following properties:
\begin{enumerate}
    \item For any $2$-commutative square of Artin stacks  
   \[ \begin{tikzcd}
&  X' \arrow[r,"A"] \arrow[d,"f'"] \arrow[d] & X \arrow[d,
"f"] \\
& Y' \arrow[r,"B"] & Y
\end{tikzcd} \]
such that $A$ is flat and of locally finite presentation. There is a natural functoriality morphism
\[ T^{*}_{X'/Y'} \rightarrow LA^{*}T^{*}_{X/Y} \]
which is an isomorphism if the square is Cartesian and either $f$ or $B$ is flat. We will usually consider this map when $Y = Y'$ and $B = id_{Y}$, in which case we will denote it by $dA_{Y}$. 
    \item If $g: Y \rightarrow Z$ is a morphism of Artin stacks locally of finite presentation and $f$ is flat then there exists a distinguished triangle 
    \[ \begin{tikzcd}
& & Lf^{*}T^{*}_{Y/Z}\arrow[dl,"+1"] & \\
& T^{*}_{X/Y} \arrow[rr] &  & T^{*}_{X/Z} \arrow[ul,"d\ol{f}"]
\end{tikzcd} \]
in $\D'_{\mathrm{qcoh}}(X,\mathcal{O}_{X})$. 
\end{enumerate}
\end{theorem}
\begin{remark}{\label{rem: contangent to tangent}}
We note that the presence of finite presentation and locally Noetherian assumptions in the above Theorem appeared when dualizing the results of \cite{Ols} and \cite{LMB} to pass from the cotangent complex to the tangent complex, where in \emph{loc.cit} it is usually just assumed that the morphisms are qcqs. For example, when formally deducing the distinguished triangle in (2) stated above, one is tasked with producing an identification 
\[ Lf^{*}R\mathcal{H}om_{\mathcal{O}_{Y}}(\mathbb{L}_{Y/Z},\mathcal{O}_{Y}) \simeq R\mathcal{H}om_{\mathcal{O}_{X}}(Lf^{*}(\mathbb{L}_{Y/Z}),Lf^{*}\mathcal{O}_{Y}) \simeq R\mathcal{H}om(Lf^{*}(\mathbb{L}_{Y/Z}),\mathcal{O}_{X}), \]
and for the first isomorphism one needs that $f$ is flat and $\bb{L}_{Y/Z}$ is pseudo-coherent (which is true since $g$ is locally of finite presentation and $Z$ is locally Noetherian, \cite[Section~8.5.29]{Illusie2005Existence}) (See \cite[Tag~0A6A (3)]{Stacks}. If $\bb{L}_{Y/Z}$ is perfect no hypothesis on $f$ is needed (e.g. if the morphisms are all smooth). Similar remarks apply to statement (1). As we will only be considering this distinguished triangle for smooth morphisms of Artin stacks, this will be sufficient for our purpose and we will implicitly do this below when citing this proposition.
\end{remark}
Let us give some flavor for how to prove this theorem in the particular case of the structure morphism of an Artin $S$-stack $X \rightarrow S$ for a scheme $S$, describing its pullback to an atlas and showing that it satisfies the cocycle condition, as in Remark \ref{rem: quasicoh}. For a morphism of schemes, $Y \rightarrow Z$, we let $L^{*}_{Y/Z}$ denote the cotangent complex, as defined in \cite[Section~II.1]{Il}. This is an object of $\mathrm{D}^{\leq 0}_{\mathrm{qcoh}}(Y,\mathcal{O}_{Y})$. For us, it will be more convenient to work with the dual notion. In particular, we define the tangent complex $T^{*}_{Y/Z} := R\mathcal{H}om_{\mathcal{O}_{Y}}(L^{*}_{Y/Z},\mathcal{O}_{Y})$, which will in turn be a complex of quasi-coherent sheaves concentrated in degrees $[0,\infty)$ on $Y$. The analogue of Theorem \ref{thm: tangcompl} in the case of regular schemes follows from the work of Illusie \cite[Section~II.2]{Il}, and we will assume this in what follows. 

Choose a smooth atlas $f:U \ra X$, and write
$R:=U\times_{X}U$ with projections $p_1,p_2:R\to U$.
Let $\Delta:U\to R$ be the diagonal, which is a section of both
projections. Writing $T^*_{R/U,p_1}$ for the relative tangent
complex with respect to $p_1$, we have
\[
T^*_{U/X}\simeq\Delta^*T^*_{R/U,p_1}.
\]
We consider the map
\[
\alpha:T^*_{U/X}\longrightarrow T^*_{U/S}
\]
 obtained by applying $L\Delta^*$ to the composite
\[
T^*_{R/U,p_1}
\longrightarrow T^*_{R/S}
\xrightarrow{dp_2} Lp_2^*T^*_{U/S},
\]
and using $p_2\circ\Delta=\operatorname{id}_U$.
Here the first arrow comes from the transitivity triangle for
$p_1$, whereas the second is the differential of $p_2$.
We then set
\[
(T^*_{X/S})_U
:=
\operatorname{Cofiber}
\bigl(T^*_{U/X}\xrightarrow{\alpha}T^*_{U/S}\bigr).
\]

Now, let us check that this description has the expected quasi-isomorphisms for maps between charts, as in Remark \ref{rem: quasicoh}. Namely, we want to show that, if we have an $S$-scheme $V \rightarrow X$ mapping smoothly to $X$ and a map $\phi: U \rightarrow V$ in the smooth topology, there is a natural map $(T^{*}_{X/S})_{U} \rightarrow \phi^{*}(T^{*}_{X/S})_{V}$ which is a quasi-isomorphism and that these satisfy a cocycle condition. To accomplish this, we note one easily sees that, using the distinguished triangles of Theorem \ref{thm: tangcompl} (2) in the case of schemes, we have a distinguished triangle
\[ \begin{tikzcd}
& & \phi^{*}T^{*}_{V/X}\arrow[dl,"+1"] & \\
& T^{*}_{U/V} \arrow[rr] &  & T^{*}_{U/X} \arrow[ul,"d\phi_{X}"]
\end{tikzcd} \]
for the above definition of $T^{*}_{U/X}$ and $T^{*}_{V/X}$. Moreover, we have a map of triangles
\[ \begin{tikzcd}
& & \phi^{*}T^{*}_{V/X}\arrow[dl,"+1"] \arrow[dd] & \\
& T^{*}_{U/V} \arrow[rr] \arrow[dd,"id"] &  & T^{*}_{U/X} \arrow[ul,"d\phi_{X}"] \arrow[dd] \\
& & \phi^{*}T^{*}_{V/S}\arrow[dl,"+1"] & \\
& T^{*}_{U/V} \arrow[rr]  &  & T^{*}_{U/S} \arrow[ul,"d\phi_{S}"]
\end{tikzcd} \]
From this, it follows that the natural commutative diagram
\[ \begin{tikzcd}
&  T^{*}_{U/X} \arrow[r] \arrow[d] & T^{*}_{U/S} \arrow[d] \\
& \phi^{*}T^{*}_{V/X} \arrow[r] & \phi^{*}T^{*}_{V/S}
\end{tikzcd} \]
induced by $\phi$ gives rise to the desired quasi-isomorphism: $(T^{*}_{X/S})_{U} \xrightarrow{\simeq} \phi^{*}(T^{*}_{X/S})_{V}$ by taking cofibers. One can then check that these isomorphisms satisfy the cocycle condition. This suggests (See \cite{Ols} for more details) that we get a well-defined object $T_{X/S}^{*} \in \D'_{\mathrm{qcoh}}(X,\mathcal{O}_{X})$ sitting in cohomological degrees $[-1,\infty)$, and that this defines the object in $\D'_{\mathrm{qcoh}}(X,\mathcal{O}_{X})$ described in Theorem \ref{thm: tangcompl}. 

Let us now work out this theory more explicitly, and derive some useful consequences for later. In particular, we recall that classically the tangent/cotangent complex controls smoothness of a morphism of schemes.
\begin{proposition}{\cite[Chapter III, Proposition~3.1.2]{Il}}{\label{prop: jacclassic}}
Let $f: X \rightarrow Y$ be a locally finitely presented morphism and $T_{X/Y}$ be the usual tangent sheaf of $X$ over $Y$. There is a natural map of complexes
\[ T_{X/Y}[0] \rightarrow T^{*}_{X/Y}, \]
which induces an isomorphism on cohomology in degree $0$. Then $f$ is smooth if and only if this map is a quasi-isomorphism and $T_{X/Y}$ is locally free. 
\end{proposition}
\begin{remark}
As in Remark \ref{rem: contangent to tangent}, the citation to \cite{Il} only gives the claim for the cotangent complex and not the tangent complex. However, one can easily check that one formally implies the other by dualizing.
\end{remark}
Moreover, we recall that, in the situation of the above proposition, the dimension of a fiber is computed by the rank of $T_{X/Y}$ around that fiber.

This reduces the geometric problem of smoothness to the more algebraic problem of computing the cohomology of the tangent complex. We would like to have a similar result for stacks. We will now work this out more explicitly.
\subsection{The Smooth Case}
Suppose that $X \rightarrow S$ is a smooth Artin $S$-stack for an affine scheme $S$. Given a point $x: S \rightarrow X$, we claim that the pullback of the tangent complex $T^{*}_{X/S}$  described in the previous section along $x$ represents the Picard tangent groupoid described in the beginning. To see this, let $\tilde{x}$ be the lift of $x$ along some smooth map $U \rightarrow X$\footnote{We note that we may always find such a lift \'etale locally on $S$, and since \'etale maps induce quasi-isomorphisms on the tangent complex via Theorem \ref{thm: tangcompl} (2) this will be sufficient.}. Using this and the above description of the tangent complex, we compute that the pullback is the complex of quasi-coherent $\mathcal{O}_{S}$-modules $L\tilde{x}^{*}\Cofiber(T^{*}_{U/X} \rightarrow T^{*}_{U/S})$. Since $U \rightarrow X$ and $U \rightarrow S$ are smooth by assumption, the quasi-coherent sheaves $T^{*}_{U/X}$ and $T^{*}_{U/S}$ are both isomorphic to the usual tangent sheaves $T_{U/X}$ and $T_{U/S}$, which both must be locally free $\mathcal{O}_{U}$-modules (Here we have used Proposition \ref{prop: jacclassic}). Therefore, we have an isomorphism  $L\tilde{x}^{*}(\Cofiber(T^{*}_{U/X} \rightarrow T^{*}_{U/S})) = \Cofiber(\tilde{x}^{*}(T_{U/X}) \rightarrow \tilde{x}^{*}(T_{U/S}))$, and an element of $\tilde{x}^{*}(T_{U/X})$ is the same as the datum of a tangent vector $S[\epsilon]/\epsilon^{2} \rightarrow U$ at $\tilde{x}$ and a trivialization of the projection of this map to $X$, with differential given by the forgetful map. Since the map $U \rightarrow X$ is smooth and $S$ is affine it follows that every tangent vector to $X$ at $x$ can be lifted to a tangent vector of $U$ at $\tilde{x}$. It follows that the pullback of the above complex represents the Picard tangent groupoid $T_{X/S,x}$ considered above. Moreover, since $X$ is smooth, by definition $U$ will also be smooth. Therefore, the dimension of $U$ over $S$ locally around $\tilde{x}$ is computed by the rank of $\tilde{x}^{*}T_{U/S}$ as an $\mathcal{O}_{S}$-module. It in turn follows that the dimension of $X$ over $S$ locally around $x$ will be computed by the Euler characteristic of $Lx^{*}T^{*}_{X/S}$ as an $\mathcal{O}_{S}$-module. We can summarize this as follows.
\begin{proposition}{\label{prop: PicardGroupoidsArtinSStacks}}
For an affine scheme $S$ and a smooth Artin $S$-stack $X$, the tangent complex $T_{X/S}^{*}$ is represented by a perfect complex of Tor-amplitude $[-1,0]$. If we pullback to a smooth chart it is represented by a length one complex of vector bundles of length $1$. The pullback of this complex to a point $x: S \rightarrow X$ is a complex of $\mathcal{O}_{S}$-modules, whose associated Picard groupoid is equivalent to the Picard groupoid $T_{X/S,x}$ defined in \S \ref{subsec: ReviewofTangentComplex}. 

Moreover, the pullback of the tangent complex $T_{X/S}^{*}$ to $x$ has Euler characteristic as a perfect complex of $\mathcal{O}_{S}$-modules equal to the dimension of $X$ locally around $x$. 
\end{proposition}
While this is nice, the real power of this theory lies in its ability to detect when $X$ is smooth over $S$ at a point $x$. As seen in Proposition \ref{prop: jacclassic}, this information will be described by the higher cohomology of the complex $T_{X/S}^{*}$.
\subsection{The Singular Case}
Given an Artin stack $X\to S$, a point
$x:S \ra X$, and a lift $\widetilde x:S \ra U$ to a
smooth chart $f:U \ra X$, the preceding construction gives
\[
Lx^*T^*_{X/S}
\simeq
\operatorname{Cofiber}\bigl(
L\widetilde x^*T^*_{U/X}
\longrightarrow
L\widetilde x^*T^*_{U/S}
\bigr).
\]
Recall that
\[
T^*_{U/X}
\simeq
\Delta^*T^*_{R/U,p_1}.
\]
Since $p_1$ is smooth, this relative tangent complex is a
locally free sheaf concentrated in degree zero.
If $U\to S$ is smooth around $\widetilde x$, then
$T^*_{U/S}$ is likewise locally represented by a vector bundle
in degree zero, and the displayed cofiber defining $(T^{*}_{X/S})_{U}$ is represented by a
two-term complex of vector bundles in degrees $[-1,0]$.

More generally, one has the following criterion.
\begin{proposition}{\cite[Proposition~17.10]{LMB}}{\label{prop: lautang}}
Let $X$ and $Y$ be Artin $S$-stacks and $X \rightarrow Y$ a $1$-morphism over $S$ that is locally of finite presentation. Then $f: X \rightarrow Y$ is smooth if and only if $T_{X/Y}^{*}$ is represented by a perfect complex on $X$ of Tor-amplitude $[-1,0]$. 
\end{proposition}
In particular, the cohomology in degree $0$ of $Lx^{*}T_{X/S}^{*}$ specifies the space of tangent vectors at a given point, while the cohomology in degree $-1$ of $Lx^{*}T_{X/S}^{*}$ controls the infinitesimal automorphisms at the point $x$. 

This theorem reduces the problem of checking smoothness of an Artin stack $X$ to computing the tangent complex. Let us conclude this section by carrying this out in some examples. Set $k$ to be an algebraically closed field. 
\begin{Example}{\label{ex: classifying}}
Let $Y$ be a variety with a a linear algebraic group $G/\Spec{k}$ acting trivially. We consider the quotient stack $[Y/G] \rightarrow \Spec{k}$. Let $\mathrm{Lie}(G)$ be the Lie algebra of $G$. We look at the usual atlas given by the projection $Y \rightarrow [Y/G]$, we claim that the pullback of $T^{*}_{[Y/G]}$ to $Y$ under this atlas is isomorphic to $\Cofiber(\mathrm{Lie}(G) \otimes \mathcal{O}_{Y} \rightarrow T^{*}_{Y})$, where the differential is defined by the action of $G$ extended $\mathcal{O}_{Y}$-linearly. In particular, the pullback by definition is given by:
\[ \Cofiber(T^{*}_{Y/[Y/G]} \rightarrow T^{*}_{Y}). \]
We consider the diagonal map $\Delta: Y \rightarrow Y \times_{[Y/G]} Y$, and let $p_{1}: Y \times_{[Y/G]} Y \rightarrow Y$ be the natural projection obtained by base-changing the map defining the atlas. If we write $Y \times_{[Y/G]} Y \simeq Y \times G$ then $p_{1}$
is projection to the first factor and $\Delta$ is the embedding into the first factor. This allows us to see that $T^{*}_{Y/[Y/G]}$ is canonically identified with the fiber of the map $T_{Y}^{*} \oplus \mathrm{Lie}(G) \otimes \mathcal{O}_{Y} \rightarrow T_{Y}^{*}$ given by projection. From here, the claim easily follows.
\end{Example}
\begin{Example}{\label{ex: BunGex}}
Fix a smooth projective curve $Y/\Spec{k}$ of genus $g$. For a connected reductive group $G/\Spec{k}$, let $X = \Bun_{G}$ denote the moduli stack of $G$-bundles on $Y$. It is well known that $\Bun_{G}$ is a smooth Artin stack over $\Spec(k)$. If we take a $k$-point $x: \Spec{k} \rightarrow X$ corresponding to a $G$-bundle $\mathcal{F}_{G}$ on $Y$, one can compute that the pullback of this complex to $x$ is quasi-isomorphic to $R\Gamma(Y,\mathcal{F}_{G} \times^{G,Ad} \mathrm{Lie}(G)[1])$ (cf. \cite[Section~4.5]{Be}). Therefore, since $Y$ is a curve, by Propositions \ref{prop: PicardGroupoidsArtinSStacks} and \ref{prop: lautang}, we can see that $\Bun_{G}$ has dimension equal to the Euler characteristic of $R\Gamma(Y,\mathcal{F}_{G} \times^{G,Ad} \mathrm{Lie}(G)[1])$ locally around $x$ which by Riemann-Roch is precisely $\dim(G)(g - 1)$ (since $\mathcal{F}_{G} \times^{G,Ad} \mathrm{Lie}(G)$ is a self-dual vector bundle in characteristic $0$ and in general $\mathrm{det}(\Lie(G))$ is the trivial character  of $G$).
\end{Example}
The appearance of the Lie algebra of $G$ in both of these examples hints at a general phenomenon relating the two, as alluded to in the introduction. To see this, let's suppose that $Y$ is a smooth projective curve over $\Spec{k}$ and that $Z$ is a smooth Artin $Y$-stack. We let $\mathcal{M}_{Z} \rightarrow \Spec{k}$ be the moduli space parameterizing sections $Y \times_{\Spec{k}} S \rightarrow Z$ over $Y$, for $S$ a $k$-scheme. Suppose that $\mathcal{M}_{Z} \rightarrow \Spec{k}$ defines an Artin stack over $\Spec{k}$\footnote{This will be true in all the examples we will consider. For general results concerning the representability of this space see \cite[Theorem~1.2]{HR}.}. If we are given a point $x: \Spec{k} \rightarrow \mathcal{M}_{Z}$ corresponding to a section $s_{x}: Y \rightarrow Z$ then we see that giving a tangent vector $D \times_{\Spec{\bb{Z}}} \Spec(k) \rightarrow \mathcal{M}_{Z}$ at $x$ is equivalent to giving a lift of the section $s_{x}$ to  a section $\tilde{s}_{x}: Y[\epsilon]/\epsilon^{2} \rightarrow Z$. This gives us an isomorphism 
\[ H^{0}(Y,Ls_{x}^{*}T^{*}_{Z/Y}) \simeq H^{0}(Lx^{*}T_{X}^{*}) \]
of vector spaces over $k$. Similarly, one has an isomorphism of $H^{-1}$ for automorphisms of the section $s_{x}$. This gives us an isomorphism of complexes of vector spaces
\[ \tau_{[-1,0]}R\Gamma(Y,Ls_{x}^{*}T^{*}_{Z/Y}) \simeq \tau_{[-1,0]}Lx^{*}T_{X}^{*} \]
over $k$ and, by considering higher order deformations, we see that a higher order deformation of $x$ is controlled by a class in $H^{1}(R\Gamma(Y,Ls_{x}^{*}T^{*}_{Z/Y}))$. In particular, if it vanishes then the space $\mathcal{M}_{Z}$ is smooth around $x$. We record this as a proposition. 
\begin{proposition}{\label{prop: secdefo}}
Let $Y$ be a smooth projective curve over an algebraically closed field $k$, $Z \rightarrow Y$ a smooth Artin $Y$-stack, and $\mathcal{M}_{Z} =: X$ the moduli stack of sections of $Z$ over $Y$. We assume that $\mathcal{M}_{Z}$ is Artin (e.g. $Z$ is quasi-separated with affine stabilizers by \cite[Theorem~1.3]{HR}) Let $x: \Spec{k} \rightarrow \mathcal{M}_{Z}$ be a point with corresponding section $s_{x}: Y \rightarrow Z$ around which $\mathcal{M}_{Z}$ is smooth. Then we have an isomorphism 
\[ \tau_{[-1,0]}R\Gamma(Y,Ls_{x}^{*}T^{*}_{Z/Y}) \simeq \tau_{[-1,0]}Lx^{*}T_{X}^{*} \]
of complexes of vector spaces over $k$. Moreover, if $H^{1}(R\Gamma(Y,Ls_{x}^{*}T^{*}_{Z/Y}))$ vanishes then $X$ is smooth around $x$.
\end{proposition}
With this in hand, let's revisit Example \ref{ex: BunGex}. 
\begin{Example}{\label{ex: sections}}
Let $Y \rightarrow \Spec{k}$ be a smooth projective curve and $G/k$ a connected reductive group, as before. We note that $\Bun_{G} = \mathcal{M}_{[Y/G]}$. Suppose that we have a $G$-torsor $\mathcal{F}_{G}$ on $Y$ corresponding to a section $s: Y \rightarrow [Y/G]$. This space is easily verified to be a smooth Artin stack so the cohomology sheaves of $T^{*}_{\Bun_{G}}$ are concentrated in degrees $[-1,0]$. Proposition \ref{prop: secdefo} tells us that to compute the pullback of the tangent complex $T^{*}_{\Bun_{G}}$ to the point defined by $\mathcal{F}_{G}$, it suffices to compute:
\[ Ls^{*}T^{*}_{[Y/G]/Y}. \]
We can do this similarly to Example \ref{ex: classifying}. In particular, we can regard the section 
\[ s: Y \rightarrow [Y/G] \]
as an atlas for $[Y/G]$ over $Y$. We then have an isomorphism: $Y \times_{[Y/G]} Y \simeq \underline{\mathrm{Aut}}_{G}(\mathcal{F}_{G})$. Arguing as in Example \ref{ex: classifying}, we can see that $T^{*}_{Y/[Y/G]}$ is identified with $\mathrm{Lie}(G) \times^{G,Ad} \mathcal{F}_{G}$. Therefore, the formula computing the pullback $Ls^{*}T^{*}_{[Y/G]/Y}$ becomes
\[ \Cofiber(\mathrm{Lie}(G) \times ^{G,Ad} \mathcal{F}_{G} \rightarrow T_{Y/Y} \simeq 0) = \mathrm{Lie}(G) \times^{G,Ad} \mathcal{F}_{G}[1], \]
which is precisely what we expect in light of Example \ref{ex: BunGex}. 
\end{Example} 
Before concluding this section, we will record one useful consequence of the above discussion that will be important for the proof of Theorem \ref{thm: stackyjacobi}. From now on, let $Y$ be an arbitrary scheme. Suppose we have a flat  morphism $f: Z_{1} \rightarrow Z_{2}$ of Artin $Y$-stacks which is locally of finite presentation. Let's suppose we have a section $s_{x}: Y \rightarrow Z_{1}$ and let $s_{y}: Y \rightarrow Z_{2}$ be the corresponding section induced by $f$. We let $Z$ denote the fiber of $f$ over $s_{y}$. Then $s_{x}$ induces a section $Y \rightarrow Z$, which we will abusively also denote by $s_{x}$. It follows from Theorem \ref{thm: tangcompl} (1)-(2) and Remark \ref{rem: contangent to tangent} that the following is true, which we record as a corollary for future use.  
\begin{corollary}{\label{cor: disttriangle}}
With notation as above, for $f: Z_{1} \ra Z_{2}$ a smooth  morphism of Artin $Y$-stacks we have the following distinguished triangle of tangent complexes
\[ \begin{tikzcd}
& & Ls_{y}^{*}T^{*}_{Z_{2}/Y}\arrow[dl,"+1"] & \\
& Ls_{x}^{*}T^{*}_{Z/Y} \arrow[rr] &  & Ls_{x}^{*}T^{*}_{Z_{1}/Y} \arrow[ul,"df_{x/y}"]
\end{tikzcd} \]
in $\D'_{\mathrm{qcoh}}(Y,\mathcal{O}_{Y})$. 
\end{corollary}
We will now review the other key player in our main Theorem, Banach-Colmez spaces. 
\section{Review of Banach-Colmez Spaces}
In the first two subsections, we will review the theory of Banach-Colmez spaces, as originally introduced in \cite{Colmez2002,LB} and later refined in \cite[Chapter~II]{FS}. In the third subsection, we will study certain stack quotients of Banach-Colmez spaces introducing the Picard $v$-groupoid mentioned in Remark \ref{rem: Picardvgroupoidlocalmodel}. Along the way, we will prove some simple lemmas related to two-term complexes of vector bundles on the Fargues-Fontaine curve that will aid us in our analysis in \S 4.

Fix $S \in \Perf$ and consider the relative algebraic Fargues-Fontaine curve $X_{S}$ and $\mathcal{E}$ a vector bundle on it. We have the following key definition.
\begin{definition}
We define the presheaf $\mathcal{H}^{0}(\mathcal{E}) \rightarrow S$ (resp. presheaf $\mathcal{H}^{1}(\mathcal{E}) \rightarrow S$) to be the functor on $\Perf_{S}$ sending $T \in \Perf_{S}$ to $H^{0}(X_{T},\mathcal{E}_{T})$ (resp. $H^{1}(X_{T},\mathcal{E}_{T})$), where $\mathcal{E}_{T}$ is the base-change of $\mathcal{E}$ to $X_{T}$. 
\end{definition}
\begin{remark}{\label{rem: whentheyarevsheaves}}
As we will discuss in more detail in Proposition \ref{prop: repofBC}, the presheaf $\mathcal{H}^{0}(\mathcal{E}) \ra S$ is always a $v$-sheaf over $S$ (in fact a locally spatial diamond) and the $v$-presheaf $\mathcal{H}^{1}(\mathcal{E})$ is also a $v$-sheaf, at least assuming that $\mathcal{E}$ has negative Harder-Narasimhan slopes after pulling back to each geometric point of $S$. In these cases, we will refer to these $v$-sheaves as Banach-Colmez spaces.
\end{remark}
Given two bundles $\mathcal{F}$ and $\mathcal{E}$ on $X_{S}$, we can identify $\mathcal{H}^{0}(\mathcal{F}^{\vee} \otimes \mathcal{E})$ with the moduli space, denoted $\mathcal{H}om(\mathcal{F},\mathcal{E})$, parameterizing maps $\mathcal{F}_{T} \rightarrow \mathcal{E}_{T}$ of $\mathcal{O}_{X_{T}}$-modules. This observation allows us to consider the following open sub-functors.  
\begin{definition}{\label{defn: InjSurj}}
Let $\mathcal{F}$ and $\mathcal{E}$ be two bundles on the algebraic Fargues-Fontaine curve $X_{S}$. We consider the following open sub-functors of $\mathcal{H}^{0}(\mathcal{F}^{\vee} \otimes \mathcal{E})$.
\begin{enumerate}
    \item We let $\mathcal{S}urj(\mathcal{F},\mathcal{E}) \subset \mathcal{H}^{0}(\mathcal{F}^{\vee} \otimes \mathcal{E})$ be the sub-functor parameterizing surjections $\mathcal{F}_{T} \rightarrow \mathcal{E}_{T}$ of $\mathcal{O}_{X_{T}}$-modules.
    \item We let $\mathcal{I}nj(\mathcal{F},\mathcal{E}) \subset \mathcal{H}^{0}(\mathcal{F}^{\vee} \otimes \mathcal{E})$ be the sub-functor parameterizing maps $\mathcal{F}_{T} \rightarrow \mathcal{E}_{T}$ whose pullback to any geometric point $\Spa(F)$ of $T$ is an injection of $\mathcal{O}_{X_{F}}$-modules. 
\end{enumerate}
If $\mathcal{F} = \mathcal{E}$ then the functor $\mathcal{S}urj(\mathcal{F},\mathcal{E})$ will identify with the presheaf parameterizing automorphisms of the vector bundle $\mathcal{E}$. We will denote this by $\mathrm{Aut}(\mathcal{E})$.
\end{definition}
\begin{remark}{\label{rem: openess}}
One can check that these both give rise to well-defined open sub-functors of $\mathcal{H}^{0}(\mathcal{F}^{\vee} \otimes \mathcal{E})$. For $\mathcal{S}urj$, see \cite[Lemma~IV.1.20]{FS}, and for $\mathcal{I}nj$ the result over a geometric point is \cite[Proposition~3.3.5]{BFH+}. To verify the claim in general, by \cite[Proposition~10.11]{Ecod} it suffices to check the claim after passing to a $v$-cover of $S$. Since $\mathrm{rank}(\mathcal{E})$, $\mathrm{rank}(\mathcal{F})$, and $\deg(\mathcal{F})$ are locally constant on $S$ by \cite[Lemma~7.2.2]{KL}, we can therefore assume they are constant. Then, by taking exterior powers, one may reduce to the case where $\mathcal{F}$ is a line bundle. Here the claim follows from the fact that the $0$-section defines a closed subspace in the Banach-Colmez space since they are separated and partially proper over the base.
\end{remark}
Similarly, if we consider $\mathcal{H}^{1}(\mathcal{F}^{\vee} \otimes \mathcal{E})$, this corresponds to the presheaf $\mathcal{E}xt^{1}(\mathcal{F},\mathcal{E})$ parameterizing extensions of $\mathcal{O}_{X_{T}}$-modules of the form:
\[ 0 \rightarrow \mathcal{E}_{T} \rightarrow \mathcal{G} \rightarrow \mathcal{F}_{T} \rightarrow 0. \]
Using this, we can define the following.
\begin{definition}
Suppose that $\mathcal{F}$ and $\mathcal{E}$ are two vector bundles on $X_{S}$ of constant rank. For $\mathcal{G}$ a vector bundle on the algebraic Fargues-Fontaine curve $X_{F}$, for $F$ an algebraically closed complete non-archimedean field, of rank equal to $rank(\mathcal{F}) + rank(\mathcal{E})$, we consider the locally closed sub-functor
\[ \mathcal{E}xt^{1}(\mathcal{F},\mathcal{E})^{\mathcal{G}} \subset \mathcal{E}xt^{1}(\mathcal{F},\mathcal{E}) \simeq \mathcal{H}^{1}(\mathcal{F}^{\vee} \otimes \mathcal{E}) \]
parameterizing extensions whose central term is isomorphic to $\mathcal{G}$ after pulling back to any geometric point\footnote{We recall that the set of isomorphism classes of vector bundles on the algebraic Fargues-Fontaine curve $X_{F}$ attached to an algebraically closed perfectoid field $F$ does not depend on the choice of algebraically closed perfectoid field (See \cite[Theorem~5.1]{GTorseursFargues}), and we are implicitly using this here.}. 
\end{definition}
\begin{remark}{\label{rem: uppersmicont}}
The claim that this defines a locally closed sub-functor is easy to see by upper semi-continuity of the Harder-Narasimhan polygon \cite[Theorem~7.4.5]{KL}. In particular, if $\mathcal{G}$ is a semi-stable vector bundle then this is actually an open sub-functor.  
\end{remark}
With these definitions out of the way, let us prove one basic lemma on smoothness that will aid us in our analysis throughout this section. Let us first recall the basic definition.
\begin{definition}{\cite[Definition~IV.1.11]{FS}}{\label{def: cohsmooth}}
Let $f: X \rightarrow Y$ be a map of Artin $v$-stacks. We say $f$ is cohomologically smooth if there exists a cohomologically smooth surjection $V \rightarrow X$ from a locally spatial diamond such that the composite $V \rightarrow X \xrightarrow{f} Y$ is separated, and that for one such map (equivalently any, using \cite[Proposition~23.13]{Ecod}) the composite is cohomologically smooth, where we note that it is a separated morphism that is representable in locally spatial diamonds, by the definition of Artin $v$-stack, so asking that it is cohomologically smooth is well-defined \cite[Definition~23.8]{Ecod}.
\end{definition}
Now we have the following basic lemma.
\begin{lemma}\label{lemma: cohdimadd}
Let $f\colon X\to Y$ and $g\colon Y\to Z$ be maps of Artin $v$-stacks.
\begin{enumerate}
  \item If $f$ and $g$ are cohomologically smooth, then so is $g\circ f$. If $f$ and $g$ are of pure
        $\ell$-dimension $d$ and $e$, then $g\circ f$ is of pure $\ell$-dimension $d+e$.
  \item If $f$ is surjective and $f$ and $g\circ f$ are cohomologically smooth, then $g$ is
        cohomologically smooth. If $f$ and $g\circ f$ are of pure $\ell$-dimension $d$ and $c$,
        then $g$ is of pure $\ell$-dimension $c-d$.
\end{enumerate}
\end{lemma}

\begin{proof}
Any map from a locally spatial diamond to an Artin $v$-stack is representable in locally spatial
diamonds, since the diagonal of the target is; so fiber products of charts are locally spatial
diamonds \cite[Proposition~13.4]{Ecod}. For maps representable in locally spatial diamonds, the
claims (including additivity of $\ell$-dimensions) are \cite[Proposition~23.13]{Ecod}. We reduce
to this case using charts as in Definition~3.7, and we compute the $\ell$-dimension of a map
$h: Y\to Z$ as $\dim_{\ell}(W\to Z)-\dim_{\ell}(W\to Y)$ for any chart $W\to Y$.

(1) Choose charts $V\to X$ for $f$ and $W\to Y$ for $g$, and put $V':=W\times_Y V$. Then
$V'\to V$ is a base change of $W\to Y$, so $V'\to V\to X$ is a separated, cohomologically
smooth surjection. Also $V'\to Z$ factors as $V'\to W\to Z$, where $V'\to W$ is a base change of
$V\to Y$; hence it is separated and cohomologically smooth. So $g\circ f$ is cohomologically
smooth, and
\[
  \dim_{\ell}(V'\to Z)-\dim_{\ell}(V'\to X)
  = \dim_{\ell}(V\to Y)-\dim_{\ell}(V\to X) + \dim_{\ell}(W\to Z)-\dim_{\ell}(W\to Y) =d+e.
\]

(2) Choose a chart $V_1\to X$ for $f$ (so $V_1\to Y$ is separated and cohomologically smooth)
and a chart $V_2\to X$ for $g\circ f$ (so $V_2\to Z$ is separated and cohomologically smooth).
Put $V:=V_1\times_X V_2$; both projections are separated, cohomologically smooth surjections.
Then $V\to V_1\to Y$ is a separated, cohomologically smooth surjection, using that $f$ is
surjective, and $V\to V_2\to Z$ is separated and cohomologically smooth. Hence $g$ is
cohomologically smooth by Definition~3.7. Finally,
$\dim_{\ell}(V\to Z)=\dim_{\ell}(V\to X)+c$ and $\dim_{\ell}(V\to Y)=\dim_{\ell}(V\to X)+d$, so $g$ has pure
$\ell$-dimension $c-d$.
\end{proof}
With this in hand, we will now study the representability of the Banach-Colmez spaces defined above over an arbitrary base $S$.
\subsection{Banach-Colmez Spaces in Families}
We recall the key smoothness and representability statements for families of Banach-Colmez spaces defined by vector bundles on the algebraic Fargues-Fontaine curve  $X_{S}$ for $S \in \Perf$ proven in \cite{FS}.
\begin{proposition}{\cite[Proposition~II.2.16,II.3.5]{FS}}{\label{prop: repofBC}}
Let $\mathcal{E}$ be a vector bundle on the relative algebraic Fargues-Fontaine curve $X_{S}$ for $S \in \Perf$ an affinoid perfectoid space. Then the following is true.
\begin{enumerate}
\item The presheaf $\mathcal{H}^{0}(\mathcal{E}) \rightarrow S$ defines a locally spatial partially proper diamond over $S$. This is cohomologically smooth over $S$ if $\mathcal{E}$ has positive Harder-Narasimhan slopes (See Definition \ref{def: HNfilt}) after pulling back to any geometric point of $S$. We will refer to this as a positive Banach-Colmez space. 
\item If $\mathcal{E}$ is a vector bundle with negative HN-slopes after pulling back to any geometric point of $S$ then the presheaf $\mathcal{H}^{1}(\mathcal{E}) \rightarrow S$ defines a locally spatial, partially proper, and cohomologically smooth diamond over $S$. We will refer to this as a negative Banach-Colmez space. 
\end{enumerate}
\end{proposition}
As a consequence of this, we can formally deduce the following Corollary from Remarks \ref{rem: openess} and \ref{rem: uppersmicont} and Proposition \ref{prop: repofBC}, where the partial properness follows directly from the moduli interpretation.
\begin{corollary}{\label{cor: extgeorep}}
Let $\mathcal{F}$ and $\mathcal{E}$ be two bundles on the relative algebraic Fargues-Fontaine curve $X_{S}$ for $S \in \Perf$ an affinoid perfectoid space. Then the following is true.
\begin{enumerate}
    \item The moduli spaces $\mathcal{I}nj(\mathcal{F},\mathcal{E}) \rightarrow S$ and $\mathcal{S}urj(\mathcal{F},\mathcal{E}) \rightarrow S$ define locally spatial partially proper diamonds over $S$. Moreover, if the slopes of $\mathcal{F}$ are strictly less than the slopes of $\mathcal{E}$ after pulling back to each geometric point then these diamonds are cohomologically smooth.
    \item If $\mathcal{F}$ and $\mathcal{E}$ are bundles such that the slopes of $\mathcal{F}$ are strictly greater than the slopes of $\mathcal{E}$ then $\mathcal{E}xt^{1}(\mathcal{F},\mathcal{E}) \rightarrow S$ defines a locally spatial, partially proper, and cohomologically smooth diamond over $S$. Moreover, if we assume that $\mathcal{F}$ and $\mathcal{E}$ are of constant rank on $X_{S}$ and let $\mathcal{G}$ be a semistable vector bundle on $X_{S}$ of rank equal to $rank(\mathcal{F}) + rank(\mathcal{E})$ after pullback to each geometric point then the same is true for the open sub-functor $\mathcal{E}xt^{1}(\mathcal{F},\mathcal{E})^{\mathcal{G}}$.
\end{enumerate}
\end{corollary}
We now discuss the situation over a geometric point. 
\subsection{Banach-Colmez Spaces over a Geometric Point}
For notational simplicity, we work throughout this section over an algebraically closed perfectoid field $F$ which is the tilt of $C$, the completed algebraic closure of $\mathbb{Q}_{p}$ (though analogous claims of course hold over any algebraically closed perfectoid field $F$). The key advantage of working over a geometric point like $\Spa(F)$ is that we can talk about the HN-filtration or Harder-Narasimhan filtration on the bundle. Namely, by \cite[Proposition~II.2.12]{FS24}, a vector bundle $\mathcal{E}$ on $X_F$ carries a unique filtration
\[
0 = \mathcal{E}_0 \subset \mathcal{E}_1 \subset \cdots \subset \mathcal{E}_k = \mathcal{E}
\]
such that $\mathcal{E}_i/\mathcal{E}_{i-1}$ is semistable of slope $\lambda_i$ and $\lambda_1 > \cdots > \lambda_k$. Moreover, the extension groups describing this filtration vanish by \cite[Proposition~II.2.5~(iii)]{FS24}, so the filtration splits. We can also classify the graded pieces: by \cite[Theorem~II.2.14]{FS24}, for every $\lambda \in \mathbb{Q}$ there exists a unique stable bundle $\mathcal{O}(\lambda)$ on $X_F$ of slope $\lambda$, and the semistable bundles of slope $\lambda$ are precisely the direct sums $\mathcal{O}(\lambda)^{\oplus n}$.

Namely, by \cite[Theorem~II.2.19]{FS}, a vector bundle $\mathcal{E}$ on $X_{F}$ carries a filtration $0 = \mathcal{E}_{0} \subset \mathcal{E}_{1} \subset \cdots \subset \mathcal{E}_{k} = \mathcal{E}$, such that $\mathcal{E}_{i}/\mathcal{E}_{i - 1}$ is semi-stable of slope $\lambda_{i}$ and $\lambda_{1} > \cdots > \lambda_{k}$. Moreover, the extension groups describing this filtration vanish, so the filtration splits. We can also classify the graded pieces. In particular, for every $\lambda \in \mathbb{Q}$, there exists a unique stable bundle $\mathcal{O}(\lambda)$ on $X$ of slope $\lambda$, and the semistable bundles of slope $\lambda$ are precisely the direct sums $\mathcal{O}(\lambda)^{\oplus n}$. Therefore, we can make the following definition.
\begin{definition}{\label{def: HNfilt}}
For $\mathcal{E}$ a bundle on $X_{F}$ as above, with Harder-Narasimhan filtration 
\[ 0 = \mathcal{E}_{0} \subset \mathcal{E}_{1} \subset \cdots \subset \mathcal{E}_{k} = \mathcal{E} \]
for any $\lambda \in \mathbb{Q}$, we define $\mathcal{E}^{\geq \lambda}$ (resp. $\mathcal{E}^{> \lambda}$) to be the subbundle given by $\mathcal{E}_{i}$ for the largest value of $i$ such that the slope of $\mathcal{E}_{i}/\mathcal{E}_{i - 1}$ is  $\geq \lambda$ (resp. $> \lambda$) and $0$ if no such subbundle exists. We set $\mathcal{E}^{< \lambda} = \mathcal{E}/\mathcal{E}^{\geq \lambda}$ and $\mathcal{E}^{\leq \lambda} = \mathcal{E}/\mathcal{E}^{> \lambda}$. 
\end{definition}
As noted above, the Harder-Narasimhan filtration is split; therefore, for any vector bundle $\mathcal{E}$ on $X_{F}$, we have a decomposition $\mathcal{E} \simeq \mathcal{E}^{< 0} \oplus \mathcal{E}^{\geq 0}$, and we can further decompose $\mathcal{E}^{\geq 0}$ as $\mathcal{E}^{= 0} \oplus \mathcal{E}^{> 0}$, where $\mathcal{E}^{= 0}$ is the semistable direct summand of slope $0$. 

We have
\[
H^0(X_F,\mathcal E)\simeq H^0(X_F,\mathcal E^{\geq 0}),
\qquad
H^1(X_F,\mathcal E)\simeq H^1(X_F,\mathcal E^{<0}).
\]
For the associated functors, the degree-zero identification holds
already as an isomorphism of presheaves:
\[
\mathcal H^0(\mathcal E)
\simeq
\mathcal H^0(\mathcal E^{\geq 0}).
\]
In degree one, we must first pass to the associated $v$-sheaf.
Write
\[
\mathcal H^1_v(\mathcal E)
:=
a_v\mathcal H^1(\mathcal E),
\]
where $a_v$ denotes sheafification for the $v$-topology.
By [FS24, Proposition II.3.4(ii)], we have
$a_v\mathcal H^1(\mathcal E^{\geq 0})=0$.
Consequently, the projection
$\mathcal E\to\mathcal E^{<0}$ induces an isomorphism
\[
\mathcal H^1_v(\mathcal E)
\simeq
\mathcal H^1(\mathcal E^{<0}),
\]
where the presheaf on the right is already a $v$-sheaf by
Proposition \ref{prop: repofBC} (2). Thus $\mathcal H^0(\mathcal E)$ and
$\mathcal H^1_v(\mathcal E)$ are locally spatial partially proper
diamonds over $\operatorname{Spa}(F)$, and
$\mathcal H^1_v(\mathcal E)$ is cohomologically smooth. We can refine this a bit further. For this, consider the perfectoid space 
\[\widetilde{\mathbb{D}}^n := \Spa(F^{\circ}[[T_1^{1/p^{\infty}},\ldots, T_n^{1/p^{\infty}}]], F^{\circ}[[T_1^{1/p^{\infty}},\ldots, T_n^{1/p^{\infty}}]]) \times_{\Spa(F^{\circ}, F^{\circ})} \Spa(F,F^{\circ})\]
given by the perfectoid open unit ball. 
\begin{proposition}{\label{prop: repovergeompoint1}}
Let $\mathcal{E}$ be any vector bundle on $X_{F}$. The following is true. 
\begin{enumerate}
    \item The Banach-Colmez space $\mathcal{H}^0(\mathcal{E}) \rightarrow \Spd(F)$ is a locally spatial partially proper diamond. If $\mathcal{E}$ has positive slopes then we can find an isomorphism:
    \[ \mathcal{H}^0(\mathcal{E})  \simeq \widetilde{\mathbb{D}}^d/\underline{\bb{Q}_{p}}^m, \]
    where $d = \deg(\mathcal{E})$ and $\underline{\bb{Q}_{p}}^m$ is a locally profinite group acting freely on $\widetilde{\mathbb{D}}^d$ for some integer $m \geq 0$. In general, there exist isomorphisms: $\mathcal{H}^0(\mathcal{E}) \simeq \mathcal{H}^{0}(\mathcal{E}^{= 0}) \times_{\Spd(F)} \mathcal{H}^0(\mathcal{E}^{>0})$ and $\mathcal{H}^{0}(\mathcal{E}^{= 0}) \simeq \underline{E}^{n}$, where $n$ is the rank of the semistable summand of slope zero. It follows that $\mathcal{H}^{0}(\mathcal{E})$ is cohomologically smooth of pure $\ell$-dimension equal to $\deg(\mathcal{E}^{> 0})$ if and only if $\mathcal{E}$ has no summand of slope $0$.
    \item If $\mathcal{E}$ has negative slopes then the diamond $\mathcal{H}^1(\mathcal{E}) \rightarrow \Spd(F)$ is locally spatial partially proper and cohomologically smooth of pure $\ell$-dimension $-\deg(\mathcal{E})$.
\end{enumerate}
\end{proposition}
\begin{proof}
The explicit presentation in part (1) is precisely \cite[Prop~3.3.2]{BFH+}. The if direction of the claim on smoothness follows from Proposition \ref{prop: repofBC} (1).  The only if direction follows since $f^{!}(\mathbb{F}_{\ell})$ for $f: \underline{E}^{n} \rightarrow \Spd(F)$ is given by the sheaf of $\mathbb{F}_{\ell}$-valued distributions on $\underline{E}^{n}$, which is far from being constant (cf. \cite[Remark~IV.1.10]{FS}). It remains to establish part (2). We have already discussed the first part of the claim above. It remains to show the claim on the $\ell$-dimension. Using the identification
\[\mathcal{H}^1(\mathcal{F} \oplus \mathcal{G}) = \mathcal{H}^1(\mathcal{F}) \times_{\Spd(F)} \mathcal{H}^1(\mathcal{G}), \]
we may assume that $\mathcal{E}$ is stable of slope $\lambda = \frac{s}{r} \in \mathbb{Q}$ for $s,r \in \mathbb{Z}$ with $(s,r) = 1$ and $r \geq 1$. Let $i_{\infty}: \{\infty\} \rightarrow X_{F}$ denote the inclusion of the closed point of $X_{F}$ defined by the untilt $C$ of $F$. Multiplying the structure sheaf $\mathcal{O}$ by a global section of $H^{0}(X_{F},\mathcal{O}_{X_{F}}(1))$ corresponding to this point gives rise to an exact sequence:
\[ 0 \rightarrow \mathcal{O} \rightarrow \mathcal{O}(1)  \rightarrow (i_{\infty})_*(C) \rightarrow 0. \]
We then tensor it by $\mathcal{O}(\lambda)$ to obtain an exact sequence:
\[ 0 \rightarrow \mathcal{O}(\lambda) \rightarrow \mathcal{O}(\lambda + 1) \rightarrow i_{\infty*}(C)^{\oplus r} \rightarrow 0. \]
Assume first that $-1 \leq \lambda < 0$ then, taking cohomology, we get a sequence of pro-\'etale sheaves:
\[0 \rightarrow \mathcal{H}^0(\mathcal{O}(\lambda + 1)) \rightarrow (\mathbb{A}_C^r)^{\diamond} \rightarrow \mathcal{H}^1(\mathcal{O}(\lambda)) \rightarrow 0. \] 
Using Part (1), Lemma \ref{lemma: cohdimadd}, \cite[Theorem~24.1,Proposition~24.4]{Ecod}, \cite[Proposition~23.13]{Ecod} for the representability, and \cite[Proposition~24.2]{Ecod} if $\lambda = -1$, this implies $\mathcal{H}^1(\mathcal{O}(\lambda))$ is representable in locally spatial diamonds and cohomologically smooth of the desired $\ell$-dimension. Similarly, if $\lambda < -1$, one gets via taking cohomology an exact sequence:
\[ 0 \rightarrow (\mathbb{A}_C^r)^{\diamond} \rightarrow \mathcal{H}^1(\mathcal{O}(\lambda)) \rightarrow \mathcal{H}^1(\mathcal{O}(\lambda + 1)) \rightarrow 0, \]
whereby the result follows from another application of Lemma \ref{lemma: cohdimadd}, \cite[Proposition~27.5]{Ecod}, \cite[Proposition~23.13]{Ecod} for the representability, and induction.
\end{proof}
We now conclude this section with some discussion of length one complexes on $X_{S}$ and their associated "Picard $v$-groupoids".
\subsection{The Picard $v$-Groupoid and Length One Complexes}
For the proof of our main Theorem, we will want to study various invariants and geometric objects attached to a length one complex of vector bundles on $X_{S}$, the relative algebraic Fargues-Fontaine curve attached to an affinoid perfectoid space $S$. In particular, we consider a two-term complex $\mathcal{E}^{*} := \{\mathcal{E}^{-1} \ra \mathcal{E}^{0}\}$ of bundles on $X_{S}$, where $\mathcal{E}^{i}$ sits in cohomological degree $i$ and write $|\mathcal{E}^{*}| := \Cofiber(\mathcal{E}^{-1} \ra \mathcal{E}^{0})$ for its cone viewed as an object in the derived category of quasi-coherent sheaves on $X_{S}$. We note that, since we are working with affinoid perfectoid spaces as test objects, it easily follows that $R\Gamma(X_{T},\mathcal{E}^{*}_{T})$ is concentrated in degrees $[-1,1]$ for all $T \in \Perf_{S}$ (See the proof of \cite[Proposition~8.7.13]{KL} or \cite[Section~II.2]{FS}). Here $\mathcal{E}^{*}_{T}$ denotes the pullback of $\mathcal{E}^{*}$ to $X_{T}$. We now have the following.
\begin{definition}{\label{defn: EulerPoincareCharacteristicofTwoTermComplex}}
For $S \in \Perf$, let $\mathcal{E}^{*} := \{\mathcal{E}^{-1} \xrightarrow{d} \mathcal{E}^{0}\}$ be a length one complex of vector bundles on $X_{S}$. We write $|\mathcal{E}^{*}| := \Cofiber(\mathcal{E}^{-1} \rightarrow \mathcal{E}^{0})$ for its cone realized as an object in the derived category of quasi-coherent sheaves on $X_{S}$, and define the following. 
\begin{enumerate}
\item For $i \in \bb{Z}$, we let $H^{i}(|\mathcal{E}^{*}|)$ denote the $i$th cohomology sheaf of $|\mathcal{E}^{*}|$ with respect to the standard $t$-structure on quasi-coherent $X_{S}$-modules. We note that this is only non-trivial in degrees $-1$ and $0$ and we have equalities
\[ \Coker(d) = H^{0}(|\mathcal{E}^{*}|) \]
and 
\[ \Ker(d) =  H^{-1}(|\mathcal{E}^{*}|) \]
for the cohomology sheaves in these degrees.
\item If $S = \Spa(F)$ is a geometric point, we write
\[ \chi(\mathcal{E}^{*}) := \deg(H^{0}(|\mathcal{E}^{*}|)) - \deg(H^{-1}(|\mathcal{E}^{*}|)), \]
where we note that, since $X_{F}$ is a Dedekind scheme (\cite[Theorem~13.5.3 (1)]{SW}), it has a well-behaved abelian category of coherent sheaves. In particular, by the description of $H^{i}(|\mathcal{E}^{*}|)$ for $i = -1,0$ provided above, these are coherent $\mathcal{O}_{X_{F}}$-modules. Moreover, they have a well-defined notion of degree by the completeness property of $X_{F}$ \cite[Proposition~2.2.1]{BFH+}. We refer to $\chi(\mathcal{E}^{*})$ as the Euler-Poincar\'e characteristic of $\mathcal{E}^{*}$.
\end{enumerate}
\end{definition}
\begin{remark}{\label{rem: MZsmisavsheaf}}
We warn the reader that it is not true that the formation of $H^{-1}(|\mathcal{E}^{*}|)$ respects base-change with respect to morphisms of test objects, as the pullback of the left exact sequence of $\mathcal{O}_{X_{S}}$-modules
\[ 0 \ra H^{-1}(|\mathcal{E}^{*}|) \ra \mathcal{E}^{-1} \ra \mathcal{E}^{0} \]
may not in general be left exact. However, it is true that the pullback is always right exact, so that the formation of $H^{0}(|\mathcal{E}^{*}|)$ respects pullback. In particular, this will tell us that $\mathcal{M}_{[Z/H]}^{\mathrm{sm}}$, as defined in Definition \ref{defn: MZHsmooth}, gives rise to a well-behaved sub $v$-stack of $\mathcal{M}_{[Z/H]}$. More precisely, we have that $\mathcal{M}_{[Z/H]}^{\mathrm{sm}}$ is a $v$-stack and there is always a monomorphism $\mathcal{M}_{[Z/H]}^{\mathrm{sm}} \ra \mathcal{M}_{[Z/H]}$ of $v$-stacks. However, we do not a priori know that this is an open immersion. This will be shown in the proof of Theorem \ref{thm:main}.
\end{remark}
We record the following basic lemma on these Euler-Poincar\'e characteristics for future use.
\begin{lemma}{\label{lemma: EulerPoincareAdditive}}
Suppose we have a distinguished triangle
\[ \xrightarrow{+1} A \ra B \ra C \xrightarrow{+1}, \]
in the derived category of quasi-coherent sheaves on $X_{F}$ for $F$ an algebraically closed complete non-archimedean field, where $A,B,C$ are represented by length one complexes of vector bundles as above. Then we have an equality
\[ \chi(B) = \chi(A) + \chi(C) \]
of Euler-Poincar\'e characteristics.
\end{lemma}
\begin{proof}
This follows from the long exact sequence of cohomology groups in the standard $t$-structure attached to the distinguished triangle (i.e. the snake lemma), together with the fact that the degree function on coherent sheaves on $X_{F}$ is additive in short exact sequences.
\end{proof}
We now want to endow the quantity $\chi(\mathcal{E}^{*})$ with some more geometric meaning. As we saw in the proof of Proposition \ref{prop: repovergeompoint1}, the degree of a vector bundle is related to the $\ell$-dimension of the Banach-Colmez spaces attached to it. This motivates the claim on the $\ell$-dimension provided in Theorem \ref{thm: jacobiancriterion}, since as mentioned in the introduction one should think of these Banach-Colmez spaces as linearizations of the moduli space $\mathcal{M}_{Z}$.

Similarly, motivated by Proposition \ref{prop: PicardGroupoidsArtinSStacks}, we will show that the quantity $\chi(\mathcal{E}^{*})$ computes the $\ell$-dimension of a certain Picard $v$-groupoid attached to the complex $\mathcal{E}^{*}$, and, as discussed in Remark \ref{rem: Picardvgroupoidlocalmodel}, this geometric object should be the linearization of the moduli space $\mathcal{M}_{[Z/H]}$, as one can see directly in certain examples (cf. Example \ref{ex: negbcspace}).

The basic definition is the following.
\begin{definition}{\label{defn: picardvgroup}}
For $S \in \Perf$ and a two-term complex of vector bundles $K := |\{\mathcal{E}^{-1} \ra \mathcal{E}^{0}\}|$, we define the following. 
\begin{enumerate}
\item We write $\mathcal{H}^{i}(K)$ for the presheaf on $\Perf_{S}$ sending $T \in \Perf_{S}$ to the group $H^{i}(R\Gamma(X_{T},|K_{T}|))$, where $K_{T}$ denotes the pullback of $K$ along $X_{T} \ra X_{S}$. 
\item We write $\mathcal{P}(|K|) \ra S$ for the $v$-prestack sending $T \in \Perf_{S}$ to the groupoid attached to the object $\tau_{\leq 0}R\Gamma(X_{T},|K_{T}|) \in \D(\bb{Z})$, as described in \S \ref{subsec: ReviewofTangentComplex} (where here we have not sheafified). 
\end{enumerate}
For each \(T\in\operatorname{Perf}_S\), the isomorphism classes of objects and the automorphism group of the neutral object of \(\mathcal P(K)(T)\) are \(H^0(R\Gamma(X_T,K_T))\) and \(H^{-1}(R\Gamma(X_T,K_T))\), respectively.
\end{definition} 
We start with the first basic Proposition on these geometric objects.
\begin{proposition}{\label{prop: repofpicard}}
Let $\mathcal{E}$ be a vector bundle on $X_{S}$, then the Picard $v$-groupoid
\[ \mathcal{P}(\mathcal{E}[1]) \rightarrow S \]
defines an Artin $v$-stack cohomologically smooth over $S$. If $S = \Spa(F)$ is a geometric point then this is pure of $\ell$-dimension $\chi(\mathcal{E}[1]) = -\deg(\mathcal{E}) = -\deg(H^{-1}(\mathcal{E}[1]))$.
\end{proposition}
\begin{proof}
By dualizing the statement of \cite[Proposition~6.2.4]{KL}, we can find an injection of vector bundles 
\[ \mathcal{E} \hookrightarrow \mathcal{O}_{X_{S}}^{\oplus m}(n) \]
for all $n$ sufficiently large and fixed $m > 0$. For $n$ sufficiently large, we can arrange that the slopes of the cokernel $\mathcal{G}$ are positive after pulling back to any geometric point of $S$. Now it suffices to check the desired statement \'etale-locally on $S$. By \cite[Proposition~II.3.4 (iii)]{FS}, we can arrange, after replacing $S$ with an \'etale covering $\tilde{S} \rightarrow S$, that $H^{1}(X_{S},\mathcal{O}_{X_{S}}^{\oplus m}(n)) = 0$. We now look at the short exact sequence
\[ 0 \ra \mathcal{E} \ra \mathcal{O}_{X_{S}}(n)^{\oplus m} \ra \mathcal{G} \ra 0. \]
and the associated distinguished triangle
\[ \xrightarrow{+1} \mathcal{E} \ra \mathcal{O}_{X_{S}}(n)^{\oplus m} \ra \mathcal{G} \xrightarrow{+1} \]
in the derived category of quasi-coherent $X_{S}$-modules. This defines for us a Cartesian diagram
\begin{equation}{\label{eqn: fiberproduct}}
 \begin{tikzcd}
\mathcal{O}_{X_{S}}(n)^{\oplus m} \arrow[r] \arrow[d] & 0 \arrow[d]  \\
\mathcal{G} \arrow[r] & \mathcal{E}[1]
\end{tikzcd} 
\end{equation}
in the derived category of quasi-coherent $\mathcal{O}_{X_{S}}$-modules.

By applying $\tau_{\leq 0}R\Gamma(X_{T},-)$ for varying $T \in \Perf_{S}$ and passing to the associated Picard $v$-groupoids, this defines a Cartesian diagram
\[ \begin{tikzcd}
\mathcal{H}^{0}(\mathcal{O}_{X_{S}}(n)^{\oplus m}) \arrow[r] \arrow[d] & S \arrow[d]  \\
\mathcal{H}^{0}(\mathcal{G}) \arrow[r] & \mathcal{P}(\mathcal{E}[1])
\end{tikzcd} \] 
of prestacks on $\Perf_{S}$, where the right vertical arrow is given by the zero section. This diagram is Cartesian after taking $T$-points for every $T \in \mathrm{Perf}_S$, since passing to the Picard groupoid attached to $\tau_{\leq 0}R\Gamma(X_T, -)$ preserves homotopy fibre products. Moreover, the map $\mathcal{H}^0(\mathcal{G}) \to \mathcal{P}(\mathcal{E}[1])$ is surjective on $T$-points for every affinoid perfectoid $T \in \mathrm{Perf}_S$. Indeed, since $n > 0$, we have $H^1(X_T, \mathcal{O}_{X_T}(n)^{\oplus m}) = 0$ by \cite[Proposition~II.2.5~(iii)]{FS}, so by the long exact sequence every class in $H^1(X_T, \mathcal{E}_T) = \pi_0\mathcal{P}(\mathcal{E}[1])(T)$ lifts to $H^0(X_T, \mathcal{G}_T)$.

Using Proposition \ref{prop: repofBC} (1) and the group structure on $\mathcal{P}(\mathcal{E}[1])$, it is easy to see that this endows $\mathcal{P}(\mathcal{E}[1])$ with the structure of an Artin $v$-stack cohomologically smooth over $S$, where the cohomologically smooth atlas is given by the map $\mathcal{H}^{0}(\mathcal{G}) \ra \mathcal{P}(\mathcal{E}[1])$ appearing above. The claim on the $\ell$-dimension easily follows from Lemma \ref{lemma: cohdimadd} and Proposition \ref{prop: repovergeompoint1} (1), where we note that $\deg(\mathcal{G}) - \deg(\mathcal{O}_{X_{F}}^{\oplus m}(n)) = -\deg(\mathcal{E})$, since $\deg(-)$ is additive in short exact sequences.
\end{proof}
While these stacky Banach-Colmez spaces might appear a bit esoteric, they actually appear quite naturally. In particular, using the above smoothness results we can verify the following.
\begin{proposition}{(\cite[Lemma~4.1]{AL}}{\label{prop: projcohsmooth} for $G = \GL_{n}$)}
Let $G$ be a connected reductive group over $\mathbb{Q}_{p}$ with parabolic subgroup $P \subset G$. Let $M$ denote the Levi factor of $P$. Consider the map $\mathfrak{q}: \Bun_{P} \rightarrow \Bun_{M}$ induced by the surjection $P \rightarrow M$. This defines a (non-representable) cohomologically smooth morphism of Artin $v$-stacks. 
\end{proposition}
\begin{proof}
For $S \in \Perf$, consider an $S$-point of $\Bun_{M}$ corresponding to an $M$-bundle $\mathcal{F}_{M}$ on $X_{S}$. It suffices to show that the fiber of $\mathfrak{q}$ over this $S$-point is a cohomologically smooth Artin $v$-stack. The fiber has a filtration coming from the filtration of the unipotent radical of $P$ by commutator subgroups (See \cite[Lemma~3.7]{HamannImaiDualizing} for more detail). The graded pieces of this filtration are of the form $\mathcal{P}(\mathcal{E}[1])$ for $\mathcal{E}$ a vector bundle on $X_{S}$ determined by $\mathcal{F}_{M}$. Here what we mean precisely is that the fiber, denoted $X$, is an iterated fibration of $v$-stacks in the $v$-stacks $Y_{i} = \mathcal{P}(\mathcal{E}_{i}[1]) \ra S$ for $i = 1,\ldots,n$ and $\mathcal{E}_{i}$ is a vector bundle on $X_{S}$ determined by $\mathcal{F}_{M}$. 

This means that there exists a sequence of morphisms 
\[ X_{n} \xrightarrow{f_{n}} \cdots \ra X_{1} \xrightarrow{f_{1}} X_0= S \]
such that $X_{n} = X$ and the maps $f_{i} \colon X_{i} \ra X_{i - 1}$ are fibrations in $Y_{i}$ for the $v$-topology on $X_{i - 1}$. In other words, there exists a $v$-surjection $T \ra X_{i - 1}$ over $S$ from a perfectoid space $T$ such that the pullback of $f_{i}$ along the surjection identifies with the natural projection  $T \times_{S} Y_{i} \ra T$. Therefore, after base change along any affinoid perfectoid $T \to X_{i-1}$, the map $f_i$ becomes the projection $T \times_S \mathcal{P}(\mathcal{E}_i[1]) \to T$, which is cohomologically smooth by Proposition \ref{prop: repofpicard}. After choosing a cohomologically smooth chart of $X_{i-1}$ by a locally spatial diamond, and the chart $\mathcal{H}^0(\mathcal{G}_i) \to \mathcal{P}(\mathcal{E}_i[1])$ from the proof of Proposition \ref{prop: repofpicard}, we reduce to a separated map representable in locally spatial diamonds. For such maps, being $\ell$-cohomologically smooth is $v$-local on the target by \cite[Theorem~1.10]{Ecod}, so each $f_i$ is cohomologically smooth. The claim for $\mf{q}$ then follows from Lemma~\ref{lemma: cohdimadd}, and the same argument justifies reducing to the fibre over an $S$-point of $\mathrm{Bun}_M$.
\end{proof}
We now turn our attention to describing the Picard $v$-groupoid of a general two-term complex $\mathcal{E}^{*} := \{\mathcal{E}^{-1} \ra \mathcal{E}^{0}\}$ of vector bundles on $X_{S}$. The key will be to use the distinguished triangle 
\begin{equation}{\label{eqn: distinguishedtriangleII}}
\begin{tikzcd}
& & \left|\mathcal{E}^{*}\right| \arrow[dl,"+1"] & \\
& \mathcal{E}^{-1} \arrow[rr] &  & \mathcal{E}^{0}. \arrow[ul]
\end{tikzcd} 
\end{equation}
We will use this to deduce the following claim.
\begin{proposition}{\label{prop: repofPicardingeneral}}
Let $\mathcal{E}^{*} := \{\mathcal{E}^{-1} \ra \mathcal{E}^{0}\}$ denote a length one complex of vector bundles on $X_{S}$, and suppose that $\mathcal{E}^{0}$ has nonnegative slopes after pulling back to each geometric point. Then the Picard $v$-groupoid $\mathcal{P}(|\mathcal{E}^{*}|)$ defines an Artin $v$-stack over $S$. If $\mathcal{E}^{0}$ has positive slopes after pulling back to each geometric point then $\mathcal{P}(|\mathcal{E}^{*}|)$ is $\ell$-cohomologically smooth over $S$. In this situation, if $S = \Spa(F)$ for an algebraically closed perfectoid field $F$, then it is $\ell$-cohomologically smooth of pure $\ell$-dimension equal to $\chi(\mathcal{E}^{*})$ over $\Spa(F)$.
\end{proposition}
\begin{proof}
We may think of the distinguished triangle (\ref{eqn: distinguishedtriangleII}) in terms of a set of Cartesian diagrams
\begin{equation}{\label{eqn: PicardGroupoids}}
\begin{tikzcd}
\mathcal{E}^{0} \arrow[r] \arrow[d] & 0 \arrow[d] \\
\text{|}\mathcal{E}^{*}\text{|} \arrow[r] \arrow[d] & \mathcal{E}^{-1}[1] \arrow[d] & \\
0 \arrow[r] & \mathcal{E}^{0}[1] & 
\end{tikzcd}
\end{equation}
in the derived category of quasi-coherent $X_{S}$-modules. By applying the functor $\tau_{\leq 0}R\Gamma(X_{T},-)$ for varying $T \in \Perf_{S}$ and looking at the induced map on Picard $v$-groupoids, the middle horizontal arrow of (\ref{eqn: PicardGroupoids}) defines for us a morphism 
\[ \eta: \mathcal{P}(|\mathcal{E}^{*}|) \ra \mathcal{P}(\mathcal{E}^{-1}[1]). \]
We claim that the map $\eta$ is representable in locally spatial diamonds, and therefore it follows by \cite[Proposition~IV.1.8 (iii)]{FS} that $\mathcal{P}(|\mathcal{E}^{*}|)$ is an Artin $v$-stack, since $\mathcal{P}(\mathcal{E}^{-1}[1])$ is an Artin $v$-stack by Proposition \ref{prop: repofpicard}.  

By applying $\tau_{\leq 0}R\Gamma(X_T, -)$ for varying $T \in \mathrm{Perf}_S$, the top Cartesian diagram of \eqref{eqn: PicardGroupoids} defines a Cartesian diagram
\[
\begin{tikzcd}
\mathcal{H}^0(\mathcal{E}^0) \arrow[r] \arrow[d] & S \arrow[d] \\
\mathcal{P}(|\mathcal{E}^*|) \arrow[r, "\eta"] & \mathcal{P}(\mathcal{E}^{-1}[1])
\end{tikzcd}
\]
of prestacks, where the right vertical arrow is the $0$-section. As in the proof of Proposition \ref{prop: repofpicard}, this holds for every $T \in \mathrm{Perf}_S$ without any vanishing hypothesis. Now let $T \in \mathrm{Perf}_S$ be affinoid and let $t \colon T \to \mathcal{P}(\mathcal{E}^{-1}[1])$ be a $T$-point. By \cite[Proposition~13.4]{Ecod}, it suffices to show that $T \times_{\mathcal{P}(\mathcal{E}^{-1}[1])} \mathcal{P}(|\mathcal{E}^*|)$ is a locally spatial diamond pro-étale locally on $T$. By \cite[Proposition~II.3.4~(ii)--(iii)]{FS}, applied to $T$, and the assumption on the slopes of $\mathcal{E}^0$, there is a pro-étale cover $\tilde{T} \to T$ with $H^1(X_{\tilde{T}}, \mathcal{E}^0_{\tilde{T}}) = 0$. By the long exact sequence attached to (\ref{eqn: distinguishedtriangleII}), $t|_{\tilde{T}}$ then lifts to a $\tilde{T}$-point of $\mathcal{P}(|\mathcal{E}^*|)$. Using the group structure on $\mathcal{P}(\mathcal{E}^{-1}[1]) \to S$, the Cartesian diagram above identifies the fibre of $\eta$ over $t|_{\tilde{T}}$ with $\mathcal{H}^0(\mathcal{E}^0)_{\tilde{T}}$, which is a locally spatial diamond by Proposition~\ref{prop: repofBC}~(1). Hence $\eta$ is representable in locally spatial diamonds with fibres $v$-locally isomorphic to $\mathcal{H}^0(\mathcal{E}^0)$. This proves the claim when $\mathcal{E}^{0}$ has nonnegative slopes.

If $\mathcal{E}^0$ has positive slopes after pulling back to each geometric point, then $\eta$ is $\ell$-cohomologically smooth. Indeed, $\eta$ is representable in locally spatial diamonds, and by the above, after base change along any $T$-point of $\mathcal{P}(\mathcal{E}^{-1}[1])$ followed by the pro-étale cover $\tilde{T} \to T$, it becomes the projection $\mathcal{H}^0(\mathcal{E}^0)_{\tilde{T}} \to \tilde{T}$. This projection is cohomologically smooth by Proposition~\ref{prop: repofBC}~(1). Since cohomological smoothness is $v$-local on the target \cite[Theorem~1.10]{Ecod}, the claim follows. Moreover, if $S = \Spa(F)$, we see that $\mathcal{P}(|\mathcal{E}^{*}|)$ is pure of $\ell$-dimension $\deg(\mathcal{E}^{0}) - \deg(\mathcal{E}^{-1})$ by combining Lemma \ref{lemma: cohdimadd}, Proposition \ref{prop: repovergeompoint1} (1), and Proposition \ref{prop: repofpicard}. Now the quantity $\deg(\mathcal{E}^{0}) - \deg(\mathcal{E}^{-1})$ is easily identified with $\chi(|\mathcal{E}^{*}|)$ by applying Lemma \ref{lemma: EulerPoincareAdditive} to the distinguished triangle (\ref{eqn: distinguishedtriangleII}), as desired.
\end{proof}
In particular, this suggests that the natural condition guaranteeing the smoothness of the Picard $v$-groupoid $\mathcal{P}(|\mathcal{E}^{*}|)$ is being representable by a two-term complex $\{\mathcal{E}^{-1} \ra \mathcal{E}^{0}\}$ of vector bundles, where $\mathcal{E}^{0}$ has positive slopes after pulling back to each geometric point. Thus, we would expect that the condition characterizing $\mathcal{M}_{[Z/H]}^{\mathrm{sm}}$ should really be given by insisting that the tangent complex $Ls^{*}T_{([Z/H])/X_{S}}$ is represented by such a complex, as opposed to the condition on $H^{0}(\mathcal{E}^{*})$ given in Definition \ref{defn: MZHsmooth}. 
The existence of such a presentation is a well-defined property
of the derived object, although positivity of the degree-zero
term depends on the chosen presentation. The following
proposition relates the existence of such a presentation to
the intrinsic condition on cohomology in Definition \ref{defn: MZHsmooth}. 
\begin{proposition}{\label{prop: smoothnessconditions}} 
For $S \in \Perf$, let $\mathcal{E}^{*} := \{\mathcal{E}^{-1} \ra \mathcal{E}^{0}\}$ be a length one complex of vector bundles on $X_{S}$. Suppose $\mathcal{E}^{0}$ has positive slopes after pulling back to each geometric point. Then the zeroth cohomology sheaf $H^{0}(|\mathcal{E}^{*}|)$ has pullback to each geometric point given by the direct sum of a torsion sheaf and a vector bundle with positive slopes (again allowing for the possibility that the free part is trivial). Conversely, suppose that $S = \Spa(F)$ is a geometric point. Then if $H^{0}(|\mathcal{E}^{*}|)/H^{0}(|\mathcal{E}^{*}|)_{\mathrm{tors}}$ has strictly positive slopes (allowing this to be zero) then $\mathcal{E}^{*}$ can be represented by a length one complex $\{\mathcal{F}^{-1} \ra \mathcal{F}^{0}\}$ of vector bundles, for $\mathcal{F}^{0}$ a bundle with positive slopes. 
\end{proposition}
\begin{proof}
For the first part, we simply note that by construction we have a surjection 
\[ \mathcal{E}^{0} \ra H^{0}(|\mathcal{E}^{*}|) \]
of $\mathcal{O}_{X_{S}}$-modules. In particular, if we pullback to each geometric point of $S$ (which is right exact) then the locally free direct summand of $H^{0}(|\mathcal{E}^{*}|)$ will admit a surjection from a vector bundle with positive slopes, which forces it to also have positive slopes.

For the converse, we first note that, when $S = \Spa(F)$, we may choose a noncanonical quasi-isomorphism $|\mathcal{E}^{*}| \simeq \Cofiber(H^{-1}(\mathcal{E}^{*}) \xrightarrow{0} H^{0}(\mathcal{E}^{*}))$ of coherent $\mathcal{O}_{X_{S}}$-modules, by virtue of the fact that the derived category of coherent $X_{F}$-modules in this case has cohomological dimension 1 (cf. \cite[Corollary~3.15]{HuybrechtsFourierMukaiTransformsinAG}). Moreover, we note that, since $H^{-1}(\mathcal{E}^{*})$ is a subsheaf of $\mathcal{E}^{-1}$ if $S = \Spa(F)$, it follows that it is actually a vector bundle (using that $X_{F}$ is a Dedekind scheme). Now, we write $H^{0}(\mathcal{E}^{*}) = \mathcal{L}_{\mathrm{free}} \oplus \mathcal{L}_{\mathrm{tors}}$, where  $\mathcal{L}_{\mathrm{free}}$ (resp. $\mathcal{L}_{\mathrm{tors}}$) is the locally free (resp. torsion) summand of $H^{0}(\mathcal{E}^{*})$. By assumption, the bundle $\mathcal{L}_{\mathrm{free}}$ has positive slopes and therefore, by taking direct sums, we reduce to showing that there exists a short exact sequence
\[ 0 \ra \mathcal{F}_{1} \ra \mathcal{F}_{2} \ra \mathcal{L}_{\mathrm{tors}} \ra 0, \]
where $\mathcal{F}_{i}$ is a vector bundle for $i = 1,2$ and $\mathcal{F}_{2}$ has positive slopes. By taking direct sums, we reduce to the case that $\mathcal{L}_{\mathrm{tors}} = \mathcal{O}_{X_{F},x}/t_{x}^{n}$, where $x$ is a closed point in $X_{F}$ and $t_{x} \in \mathcal{O}_{X_{F},x}$ is a uniformizing element in the local ring of $X_{F}$ at $x$. The desired claim now follows from the obvious short exact sequence 
\[ 0 \ra \mathcal{O}_{X_{F}} \ra \mathcal{O}_{X_{F}}(n) \ra \mathcal{O}_{X_{F},x}/t_{x}^{n} \ra 0.\]
\end{proof}
The previous two Propositions imply the following Corollary, which supports our expectation that the Picard $v$-groupoid of the tangent complex is an infinitesimal linear model for the $v$-stack $\mathcal{M}_{[Z/H]}^{\mathrm{sm}}$. 
\begin{corollary}{\label{cor: PicardGroupoidSmooth}}
For $S \in \Perf$ and $Z$, a scheme smooth quasi-projective over $X_{S}$ together with an action of a smooth linear algebraic group $H/\bb{Q}_{p}$, let $s_{x}: X_{F} \ra [Z/H]$ be a section of $\mathcal{M}_{[Z/H]}$ corresponding to a geometric point $x: \Spa(F) \ra \mathcal{M}_{[Z/H]}$. Then, if $x$ lies in the sub $v$-stack $\mathcal{M}_{[Z/H]}^{\mathrm{sm}}$ it follows that  $\mathcal{P}(|Ls_{x}^{*}T^{*}_{([Z/H])/X_{S}}|) \ra \Spa(F)$, the Picard $v$-groupoid of the tangent complex, is an $\ell$-cohomologically smooth Artin $v$-stack of pure $\ell$-dimension $\chi(Ls_{x}^{*}T^{*}_{([Z/H])/X_{S}})$. 
\end{corollary}
With all the necessary background now reviewed, we turn our attention to the proof of the main result.
\section{A Jacobian Criterion for Artin $v$-stacks}
\subsection{Proof of the Main Theorem}{\label{subsec: proofofMainTheorem}}
In this section, we will deduce Theorem \ref{thm: stackyjacobi} from Theorem \ref{thm: jacobiancriterion}. 
Before doing this, let us mention a particularly nice example where a result similar to Theorem \ref{thm: stackyjacobi} can be easily verified.  
\begin{Example}{\label{ex: negbcspace}}
Let $S$ be a perfectoid space and $\mathcal{E}^{\text{gm}} \rightarrow X_{S}$ be the geometric realization of a vector bundle $\mathcal{E}$ of rank $n$ on $X_{S}$, the relative algebraic Fargues-Fontaine curve attached to an affinoid perfectoid space $S \in \Perf$. We then consider the Artin $X_{S}$-stack $[X_{S}/\mathcal{E}^{\text{gm}}] \rightarrow X_{S}$. We claim that $\mathcal{M}_{[X_{S}/\mathcal{E}^{\text{gm}}]}$ is given by $\mc{P}(\mathcal{E}[1]) \ra S$, the Picard $v$-groupoid attached to the complex $\mathcal{E}[1]$ of bundles on $X_{S}$, as in Definition \ref{defn: picardvgroup}. To see this, note that the groupoid of sections over $X_S$
is the groupoid of $\mathcal{E}$-torsors on $X_S$. This is
naturally equivalent to the Picard groupoid associated to
$\tau_{\leq 0}R\Gamma(X_S,\mathcal{E}[1])$. Its isomorphism
classes of objects are parametrized by $H^1(X_S,\mathcal{E})$,
and the automorphism group of each object is
$H^0(X_S,\mathcal{E})$. These identifications are functorial
under base change in $S$, giving the claimed equivalence.

Let's see what the pullback of the tangent complex looks like in this case. In particular, suppose we have an $S$-point of $\mathcal{M}_{[X_{S}/\mathcal{E}^{\mathrm{gm}}]}$. This determines an atlas
\[ X_{S} \rightarrow [X_{S}/\mathcal{E}^{\text{gm}}]\]
such that the equivalence relation $R = \mathcal{E}^{\text{gm}}$. Then computing as in Example \ref{ex: classifying}, we see that the pullback of the tangent complex $Ls_{x}^{*}T^{*}_{[X_{S}/\mathcal{E}^{\text{gm}}]/X_{S}}$ is given by
\[\{ \mathcal{E} \rightarrow 0 \} \simeq \mathcal{E}[1]. \]
In particular, assuming $S$ is affinoid, the complex
$R\Gamma(X_S,\mathcal{E}[1])$ has cohomology concentrated
in degrees $-1$ and $0$, given by $H^0(X_S,\mathcal{E})$
and $H^1(X_S,\mathcal{E})$, respectively. Therefore, $\mathcal{M}_{[X_{S}/\mathcal{E}^{\text{gm}}]}$ is isomorphic to the Picard $v$-groupoid of its associated tangent complex, so we can think of this example as the "linear" case of Theorem \ref{thm: stackyjacobi}, just as $Z = \mathcal{E}^{\mathrm{gm}} \ra X_{S}$ gives the "linear" case of Theorem \ref{thm: jacobiancriterion}.
\end{Example}

Let $S\in\Perf$, and consider the presentation
\[
 \mc{X} =[Z/\GL_n],\qquad \mathcal{M} = \mathcal{M}_{\mc{X}},
\]
with $Z$ smooth and quasi-projective over $X_{S}$. For the rest of the section, we assume moreover that the natural morphism
\[
  \mathcal X=[Z/\mathrm{GL}_n]
  \longrightarrow [X_S/\mathrm{GL}_n]
\]
is representable in smooth quasi-projective schemes.
We will remove this additional hypothesis in the final
reduction to the general case. The forgetful map gives
\[
 \mc{X} \longrightarrow [X_{S}/\GL_{n}] \]
 \[ \mc{M} \longrightarrow\Bun_{\GL_n,S}.
\]

For integers $m\geq0$ and $r>n$, consider the $X_{S}$-schemes
\[
 V_{m,r}^{\mathrm{gm}},
 \qquad
 G_{m,r} :=\Gr_{X_{S}}(V_{m,r},n),
\]
where the Grassmannian parametrizes locally free rank-$n$ quotients and $V_{m,r} :=(\mc{O}_{X_{S}}(-m)^{\oplus r})$. We consider the map $G_{m,r}\to [X_{S}/\GL_n]$ defined by taking the universal vector bundle defined by the quotient.

We write $\mc{Q}_{m,r}\subset \mc{M}_{G_{m,r}}$ for the subfunctor parameterizing, for $T \in \Perf_{S}$, quotients
\begin{equation}{\label{eqn: Quotientbyq}}
 q:V_{m,r}|_{X_T}\twoheadrightarrow \mc{E}
\end{equation}
such that
\begin{equation}\label{eq:slope-bound}
 \mu_{\min}(\mc{E}|_{X_F})>-m
\end{equation}
for every geometric point $\Spa(F)\to T$, where $\mu_{\min}$ denotes the smallest HN-slope. Similarly, we write $\mu_{\max}$ for the largest slope of a bundle. This is an open subfunctor by the usual upper semi-continuity results for families of slope polygons on Fargues-Fontaine curves \cite[Theorem~7.4.5]{KL}.

\begin{proposition}{\label{prop:quot}}
The maps
\[
 \pi_{m,r}:\mc{Q}_{m,r}\longrightarrow\Bun_{\GL_n,S}
\]
are representable in locally spatial diamonds, separated, and cohomologically smooth. As $m$ and $r$ vary, they form a jointly $v$-surjective map.
\end{proposition}
\begin{proof}
These are the cohomologically smooth charts of \cite[Example~IV.4.7]{FS}, with $N=-m$. Here they describe the cohomologically smooth open subspace $\mathcal{M}_{G_{m,r}}^{\mathrm{sm}} \subset \mathcal{M}_{G_{m,r}}$ of Definition \ref{def: Mzsm} explicitly. In particular, if $q$ is a surjection as in (\ref{eqn: Quotientbyq}) it is the locus where the greatest slope of $\mathrm{Ker}(q)$ is smaller than the smallest slope of the quotient $\mathcal{E}$. We need to show that $\mathcal{Q}_{m,r} \subset \mathcal{M}_{G_{m,r}}^{\mathrm{sm}}$, by virtue of this being an open immersion. 

To see this, we let $K=\ker(q)$ be the kernel of a map $q$ (\ref{eqn: Quotientbyq}) corresponding to a geometric point of $\mathcal{Q}_{m,r}$. Since $K$ is a subbundle of the semistable bundle $\mc{O}(-m)^{\oplus r}$, we have an inequality of slopes
\[
 \mu_{\max}(K)\leq-m<\mu_{\min}(\mc{E}),
\]
where the second inequality is the condition (\ref{eq:slope-bound}). Thus, the condition required is automatic on \eqref{eq:slope-bound}.

Fix $T \in \Perf_{S}$ with a map $T \rightarrow \Bun_{\GL_{n},S}$ corresponding to a vector bundle on $X_{T}$, the pullback of $\pi_{m,r}$ along this $T$-point is the open (by Remark \ref{rem: openess}) subdiamond 
\begin{equation}\label{eq:fiber}
 \mathcal{S}\mathrm{urj}(V_{m,r},\mc{E})
 \subset \mc{H}\mathrm{om}(V_{m,r},\mc{E}\bigr)
 =\mathcal{H}^{0}\bigl(\mc{E}(m)^{\oplus r}\bigr),
\end{equation}
restricted to the open subset defined by \eqref{eq:slope-bound}. The vector bundle $\mathcal{E}(m)^{\oplus r}$ on the right has strictly positive slopes by our assumption (\ref{eq:slope-bound}), so $\mathcal{H}^{0}\bigl(\mc{E}(m)^{\oplus r}\bigr)$ is $\ell$-cohomologically smooth by Proposition \ref{prop: repofBC} (1). The claim now follows from the fact that the subfunctor of surjective maps is an open subset follows, by Remark \ref{rem: openess}.

The surjectivity assertion follows by presenting a vector bundle, $v$-locally on $T$, as a quotient of $\mathcal{O}_{X_{T}}(-m)^{\oplus r}$ for $r$ and $m$ sufficiently large (\cite[Theorem~II.2.6]{FS},\cite[Proposition~6.2.4]{KL}). 
\end{proof}
We define
\begin{equation}\label{eq:Y}
 Y_{m,r} :=G_{m,r}\times_{[X_{S}/\GL_{n}]} \mc{X}.
\end{equation}
Since the map $\mc{X} \ra [X_{S}/\GL_{n}]$ is smooth quasi-projective by assumption and $G_{m,r}$ is a smooth projective $X_{S}$-scheme, this is also a smooth quasi-projective $X_{S}$-scheme. We set
\begin{equation}\label{eq:U}
 \mc{U}_{m,r} :=\mc{Q}_{m,r}\times_{\Bun_{\GL_n,S}} \mc{M},
 \qquad
 f_{m,r}:\mc{U}_{m,r}\longrightarrow \mc{M}.
\end{equation}
Thus, there is a Cartesian diagram of stacks over $X_{S}$
\begin{equation}{\label{eqn: fullcartesiandiagram}}
\begin{tikzcd}[column sep=large,row sep=large]
 Y_{m,r}\arrow[r]\arrow[d] & \mc{X}\arrow[d]\\
 G_{m,r}\arrow[r] & \left[X_{S}/\GL_{n}\right]
\end{tikzcd}
\end{equation}
inducing a Cartesian diagram
\begin{equation}{\label{eqn: FullCartesianDiagram}}
\begin{tikzcd}[column sep=large,row sep=large]
 \mc{U}_{m,r}\arrow[r,"f_{m,r}"]\arrow[d] & \mc{M} \arrow[d]\\
 \mc{Q}_{m,r}\arrow[r,"\pi_{m,r}"] & \Bun_{\GL_n,S},
\end{tikzcd}
\end{equation}
by taking moduli space of sections and then further pulling back along the open subfunctor $\mc{Q}_{m,r} \ra \mathcal{M}_{G_{m,r}}$. In particular, $\mc{U}_{m,r}$ is the open subfunctor of $\mc{M}_{Y_{m,r}}$ on which the associated point of $\mathcal{M}_{G_{m,r}}$ satisfies \eqref{eq:slope-bound} after pulling back to each geometric point. 

\begin{proposition}\label{prop:atlas}
Each $\mc{U}_{m,r}$ is a locally spatial diamond over $S$. The maps $f_{m,r}$ are representable in locally spatial diamonds, separated, cohomologically smooth, and jointly $v$-surjective. In particular, $\mc{M}$ is an Artin $v$-stack.
\end{proposition}
\begin{proof}
The first assertion follows from \cite[Theorem~IV.4.2]{FS} applied to $Y_{m,r} \ra X_{S}$ where it was explained above that this morphism is smooth quasi-projective. The properties of $f_{m,r}$ follow from Proposition~\ref{prop:quot} and the Cartesian diagram (\ref{eqn: FullCartesianDiagram}).
\end{proof}

\begin{theorem}\label{thm:main}
Let $S\in\Perf$, let $H/\bb{Q}_{p}$ be a linear algebraic group, and let $Z$ be a smooth quasi-projective $X_S$-scheme with an $H$-action. Put $\mc{X}=[Z/H]$ and $\mc{M} =\mc{M}_{\mc{X}}$. Define $\mc{M}^{\mathrm{sm}}$ as in Definition \ref{defn: MZHsmooth}. Then $\mc{M}$ is an Artin $v$-stack, $\mc{M}^{\mathrm{sm}} \subset \mc{M}$ is open, and $\mc{M}^{\mathrm{sm}} \to S$ is cohomologically smooth. At a geometric point represented by $s: X_{F} \ra \mc{X}$, its $\ell$-dimension is
\begin{equation}\label{eq:main-dimension}
 \deg(H^0(Ls^*T^{*}_{\mc{X}/X_S}))
 -\deg(H^{-1}(Ls^*T^{*}_{\mc{X}/X_S})).
\end{equation}
\end{theorem}

\begin{proof}
We first prove the theorem for $H=\mathrm{GL}_n$ under
the additional hypothesis on the representability of $\mc{X} \ra [X_{S}/\GL_{n}]$ imposed above. Proposition~\ref{prop:atlas} proves that $\mc{M}$ is an Artin $v$-stack in this case, using the charts $f_{m,r}$ of  Proposition~\ref{prop:atlas}. We start with the following key lemma. 
\begin{lemma}{\label{lemma: keylemma}}
For all $m,r \geq 0$ with  $r > n$, the pullback of $\mc{M}^{\mathrm{sm}} \ra \mc{M}$ along $f_{m,r}$ agrees with $\mc{U}_{m,r}^{\mathrm{sm}} := \mc{M}_{Y_{m,r}}^{\mathrm{sm}} \times_{\mc{M}_{Y_{m,r}}} \mc{U}_{m,r}$.
\end{lemma}
\begin{proof}
It suffices to show the claim at the level of geometric points since the substack defining $(-)^{\mathrm{sm}}$ is given by a condition on the pullbacks of a section to geometric points. We fix a geometric point $s: X_{F} \ra \mc{X}$ and a geometric point $\tilde{s}: X_{F} \ra Y_{m,r}$ lying over it which defines a point of the open subfunctor $\mc{U}_{m,r} \subset \mc{M}_{Y_{m,r}}$. We write $\mathcal{E}$ for the associated quotient bundle and consider the tangent complexes
\begin{equation}{\label{eqn: tangentcomplexes}}
 C=Ls^*T^{*}_{\mc{X}/X_{S}},\qquad
 B=\widetilde s^*T^{*}_{Y_{m,r}/X_{S}},\qquad
 A=\mathrm{Hom}_{\mathcal{O}_{X_{F}}}(V_{m,r}|_{X_F},\mc{E})=\mc{E}(m)^{\oplus r}.
\end{equation}
Here $C$ has amplitude $[-1,0]$ by virtue of the fact that $\mc{X}/X_{S}$ is a smooth Artin stack over $X_{S}$, $B$ is a vector bundle by the representability of $Y_{m,r}$, and $A$ is positive by our restriction to $\mc{U}_{m,r} \subset \mc{M}_{Y_{m,r}}$.

The relative tangent bundle of the map $G_{m,r}\to [X_{S}/\GL_n]$ given by the universal quotient of $V_{m,r}$ is precisely $A$. This was already seen in the proof of Proposition \ref{prop:quot}. The fiber of $G_{m,r}\to [X_{S}/\GL_n]$ over the section corresponding to a fixed bundle $\mc{E}$ is the $X_{S}$-scheme parameterizing surjections from $V_{m,r}$ to $\mc{E}$. Base change (Theorem \ref{thm: tangcompl} (1)) applied to the Cartesian diagram (\ref{eqn: fullcartesiandiagram}) and the distinguished triangle attached to $Y_{m,r} \ra \mathcal{X}$ (Theorem \ref{thm: tangcompl})  (cf. \label{cor: disttriangle} and Remark \ref{rem: contangent to tangent}) gives a distinguished triangle
\begin{equation}\label{eq:triangle}
\xrightarrow{+1} A \rightarrow B \rightarrow C \xrightarrow{+1} .
\end{equation}
Taking cohomology sheaves of the long exact sequence gives
\begin{equation}\label{eq:five-term}
 0\longrightarrow H^{-1}(C)\longrightarrow A
 \longrightarrow B\longrightarrow H^0(C)\longrightarrow0.
\end{equation}
By definition, a section $s$ lies in $\mathcal{M}^{\mathrm{sm}}$ if $H^{0}(C)$ has torsion free part with strictly positive slopes. Moreover, a section $\tilde{s}$ lies in $\mc{U}_{m,r}^{\mathrm{sm}}$ if $B$ has strictly positive slopes. It is easy to see that if $B$ has strictly positive slopes then so must $H^{0}(C)$ since it admits a surjection from $B$. Conversely, if $H^{0}(C)$ has positive slopes then $B$ must have positive slopes, since the image of $A$ inside $B$ must be a vector bundle with strictly positive slopes (as it admits a surjection from $A$ which has strictly positive slopes). 
\end{proof}
By Theorem \ref{thm: jacobiancriterion}, we note that the open subspaces $\mc{U}_{m,r}^{\mathrm{sm}} \subset \mathcal{M}_{Y_{m,r}}^{\mathrm{sm}}$ are cohomologically smooth over $S$. Moreover, by Lemma \ref{lemma: keylemma}, we know that the pullback of $\mathcal{M}^{\mathrm{sm}} \subset \mathcal{M}$ is precisely the open subspace $\mc{U}_{m,r}^{\mathrm{sm}} \subset \mathcal{M}_{Y_{m,r}}$. Since the $f_{m,r}$ form a $v$-cover, this implies that $\mathcal{M}^{\mathrm{sm}} \subset \mathcal{M}$ is an open subspace using \cite[Proposition~10.11 (i)]{Ecod}, as desired. 
Moreover, the $\mc{U}_{m,r}^{\mathrm{sm}} \ra \mathcal{M}^{\mathrm{sm}}$ give a $\ell$-cohomologically smooth atlas by an $\ell$-cohomologically smooth space, showing that $\mathcal{M}^{\mathrm{sm}}$ is an $\ell$-cohomologically smooth Artin $v$-stack over $S$.

It remains to compute the $\ell$- dimension of $\mathcal{M}^{\mathrm{sm}}$. We fix a geometric point $s: X_{F} \ra \mc{X}$ defining a geometric point $x: \Spa(F) \ra \mathcal{M}^{\mathrm{sm}}$. We lift this to get a section $\tilde{s}: X_{F} \ra Y_{m,r}$ which also defines a point $\tilde{x}: \Spa(F) \ra \mc{U}_{m,r}^{\mathrm{sm}}$ using Lemma \ref{lemma: keylemma}. We use the notation for the associated tangent complexes as in (\ref{eqn: tangentcomplexes}).

The Jacobian criterion (Theorem \ref{thm: jacobiancriterion}) tells us that the $\ell$-dimension of $\mc{U}_{m,r}^{\mathrm{sm}}$ over $S$ around $\tilde{x}$ is given by
\[
\deg(B).
\]
The description of the fibers of \eqref{eq:fiber}  tells us that the $\ell$-dimension of $f_{m,r}$ locally around $\tilde{x}$ is given by
\[
\deg(A)
 =r\bigl(\deg(\mc{E})+nm\bigr).
\]
Lemma \ref{lemma: cohdimadd} consequently tells us that the $\ell$-dimension of $\mc{M}^{\mathrm{sm}}$ over $S$ around $s$ is given by
\[
\deg(B)-\deg(A)
 =\deg(H^0(C))-\deg(H^{-1}(C)),
\]
where the last equality is simply Lemma \ref{lemma: EulerPoincareAdditive} applied to the distinguished triangle (\ref{eq:triangle}).

Now, for a general $H$, we can choose a closed immersion $H\hookrightarrow\GL_n$. The equivalence
\begin{equation}{\label{eqn: charts}}
 [Z/H]\simeq[(Z\times^H\GL_n)/\GL_n]
\end{equation}
identifies the relative tangent complexes as well as the section stacks. However, we do not necessarily know that $Z\times^H\GL_n$
is quasi-projective over $X_{S}$. Nonetheless, we may still construct charts $\mc{U}_{m,r} \ra \mathcal{M}_{[Z/H]}$ just as above; however, we do not necessarily know that $\mc{U}_{m,r}$ are representable in locally spatial diamonds over $S$ they are a priori just separated $v$-sheaves. We note, by \cite[Theorem~1.10]{Ecod}, that the property of being $\ell$-cohomologically smooth of a given $\ell$-dimension is $v$-local on $S$. By \cite[Proposition~13.4 (iii)-(iv)]{Ecod}, we know that representability in locally spatial diamonds may be checked after replacing $S$ by a pro-\'etale cover. Choose the affinoid pro-\'etale cover $S'\to S$ provided
by Proposition \ref{lem:v-local-linearization}, and put
\[
  W=(Z\times^H\mathrm{GL}_n)\times_{X_S}X_{S'}.
\]
The construction in the proof of Proposition \ref{lem:v-local-linearization} gives
a normal quasi-projective $\mathrm{GL}_n$-scheme $W_1$
of finite type over $\mathbb Q_p$, a scheme $T_1$ on
which $\mathrm{GL}_n$ acts trivially, and an equivariant
identification
\[
  W \simeq W_1\times_{T_1}X_{S'}.
\]
By \cite[Lemma~2.9]{BrionActions}, applied to the irreducible components of $W_{1}$,
a positive tensor power of an ample line bundle on
$W_1$ admits a $\mathrm{GL}_n$-linearization.
Its pullback to $W$ is therefore a
$\mathrm{GL}_n$-linearized relatively ample line bundle.

For any $X_{S'}$-scheme $T$ and any
$\mathrm{GL}_n$-torsor $P\to T$, the associated bundle
$P\times^{\mathrm{GL}_n}W$ is a smooth scheme over $T$:
this can be checked on a Zariski cover trivializing $P$.
The linearized line bundle descends to a relatively
ample line bundle on this associated bundle, which
is consequently quasi-projective over $T$.
Thus
\[
  [W/\mathrm{GL}_n]\longrightarrow[X_{S'}/\mathrm{GL}_n]
\]
is representable by smooth quasi-projective schemes.
The case already proved therefore applies over $S'$,
and the preceding descent argument gives the theorem
over $S$.
\end{proof}
Now, let's show the utility of this result in some applications. 
\subsection{Applications}{\label{subsec: applications}}
First, let's start by reproving a result shown in Fargues-Scholze. Throughout, $F$ will denote an algebraically closed perfectoid field. We that since $\ell$-cohomological smoothness is $v$-local on the target \cite[Theorem~1.10]{Ecod} this will also imply the claims for the stacks defined over $\Spd(\ol{\bb{F}}_{p})$.
\begin{theorem}{\cite[Theorem~IV.1.19]{FS}}{\label{thm: bungsmooth}}
Let $G/\mathbb{Q}_{p}$ be a connected reductive group. Let $\Bun_{G} \rightarrow \Spd(F)$ be the moduli stack parameterizing $G$-bundles on the relative algebraic Fargues-Fontaine curves $X_{S}$, for $S \in \Perf_{F}$. Then $\Bun_{G}$ is an Artin $v$-stack cohomologically smooth over $\Spd(F)$ of pure $\ell$-dimension $0$.
\end{theorem}
\begin{proof}
We apply Theorem \ref{thm: stackyjacobi} with $Z = X_{F}$ and $H = G$. In particular, $\mathcal{M}_{[Z/H]} \simeq \Bun_{G}$ in this case. Given a section $s: X_{F} \rightarrow [X_{F}/G]$ corresponding to a $G$-torsor $\mathcal{F}_{G}$, we compute as in Example \ref{ex: BunGex} that $Ls^{*}T^{*}_{[Z/H]/X_{F}} \simeq \mathrm{Lie}(G) \times^{G,Ad} \mathcal{F}_{G}[1]$, and this satisfies the smoothness condition since there are only non-zero cohomology sheaves in degree $-1$. Moreover, we see that $\deg(\mathrm{Lie}(G) \times^{G,Ad} \mathcal{F}_{G}) = 0$ because $G$ is reductive and therefore the bundle $\mathrm{Lie}(G) \times^{G,Ad} \mathcal{F}_{G}$ is self-dual. This gives the desired claim.
\end{proof}
Let's now consider a slightly more interesting example. Let $G/\mathbb{Q}_{p}$ be a connected reductive group as before and assume that $G$ is quasi-split for simplicity. Fix a maximal split torus, maximal torus, and Borel $A \subset T \subset B \subset G$. Let $P \subset G$ be a proper parabolic subgroup with Levi factor $M$. We recall \cite[Theorem~1.1]{Vi} that the underlying topological space of $\Bun_{M}$, denoted $|\Bun_{M}|$, is isomorphic to the Kottwitz set $B(M)$ of $M$ (where here we take the normalization that sends the line bundle $\mathcal{O}(-1)$ to the isocrystal of slope $1$), equipped with the partial ordering given by the maps
\[ \nu \times \kappa: B(M) \rightarrow (X_*(M_{\overline{\mathbb{Q}}_{p}}) \otimes \mathbb{Q})^{\Gamma,+} \times \pi_{1}(M)_{\Gamma} \]
\[ b \mapsto (\nu_{b},\kappa(b)) \]
where $\Gamma := \Gal(\overline{\mathbb{Q}}_{p}/\mathbb{Q}_{p})$ is the absolute Galois group, $(X_*(M_{\overline{\mathbb{Q}}_{p}}) \otimes \mathbb{Q})^{\Gamma,+}$ is the set of Galois-invariant dominant rational cocharacters of $M$, and $\pi_{1}(M)$ is Borovoi's fundamental group of $M_{\overline{\mathbb{Q}}_{p}}$ (See \cite[Theorem~1.8,Theorem 1.15]{RapRicFIsocrystals} for a description of the maps $b \mapsto \nu_{b}$ and $b \mapsto \kappa(b)$ together with a description of their basic properties). Namely, we say that $b \succeq b'$ if $\nu_{b} - \nu_{b'}$ is a $\mathbb{Q}_{\geq 0}$-linear combination of positive coroots and $\kappa(b) = \kappa(b')$. Therefore, the connected components of $\Bun_{M}$, the moduli stack of $M$-bundles on the relative algebraic Fargues-Fontaine curve $X_{S}$ for $S \in \Perf_{F}$ are parametrized by the minimal elements $B(M)_{basic} \xrightarrow{\simeq} \pi_{1}(M)_{\Gamma}$ in this partial ordering on $B(M)$. In particular, each basic element defines an open and dense Harder-Narasimhan strata in the associated connected component. Given $\nu \in B(M)_{basic}$, we write $\Bun_{M}^{(\nu)}$ for the associated connected component. We let $\Bun_{P}$ denote the moduli space of $P$-bundles on $X_{S}$ for $S \in \Perf_{F}$, and let $\Bun_{P}^{(\nu)}$ denote the pre-image of this connected component under the natural map $\mathfrak{q}: \Bun_{P} \rightarrow \Bun_{M}$. Using the proof of Proposition \ref{prop: projcohsmooth} and \cite[Proposition~23.11]{Ecod}, we can show that this morphism is open with connected fibers, which implies that $\Bun_{P}^{(\nu)}$ are the connected components of $\Bun_{P}$. Given a $P$-bundle $\mathcal{G}_{P}$ on $X_{F}$ and a character $\chi \in X^{*}(P)$ in the character lattice of $P$, we have an associated line bundle $\chi_{*}(\mathcal{G}_{P})$. The mapping 
\[ X^{*}(P) \rightarrow \mathbb{Q} \]
\[ \chi \mapsto \deg(\chi_{*}(\mathcal{G}_{P})) \]
defines an element $\mu(\mathcal{G}_{P}) \in X_*(A)_{\mathbb{Q}}$, the set of rational cocharacters of $A$ (Here we have implicitly used the identification $X^{*}(P) \otimes \bb{Q} \simeq X^{*}(M^{\mathrm{ab}}) \otimes \bb{Q} \simeq X^{*}(Z(M)) \otimes \bb{Q}$ (See \cite[Section~5.1]{GTorseursFargues})). We refer to this as the slope cocharacter of $\mathcal{G}_{P}$. For all $\nu \in B(M)_{basic}$, we can consider the corresponding $M$-bundle $\mathcal{F}_{M}^{\nu}$ and then we define the integer
\[ d_{\nu} := \langle 2\rho_{G}, \mu(\mathcal{F}_{M}^{\nu} \times^{M} P) \rangle,  \]
where $2\rho_{G} \in X^{*}(A)$ is the sum of all positive roots of $G$ defined by the choice of Borel. 
\begin{Example}
Suppose that $G = \GL_{n}$, $B$ is the upper triangular Borel, $h_{i}$ are positive integers such that $\sum_{i = 1}^{k} h_{i} = n$, and $P$ is the upper block triangular parabolic with Levi factor $M = \GL_{h_{1}} \times \cdots \times \GL_{h_{k}}$. Then $\nu \in B(M)_{basic}$ is specified by a tuple of integers $(d_{1},\ldots,d_{k})$, where $d_{i}$ are the degrees of the semistable bundles of rank $h_{i}$ defining the $M$-bundle $\mathcal{F}_{M}^{\nu}$. Then we have an equality:
\[ d_{\nu} = \sum_{i > j} (h_{i}d_{j} - h_{j}d_{i}). \]
In particular, if the slopes of the $M$-bundle are strictly increasing $-d_{\nu}$ is positive, and if they are strictly decreasing $-d_{\nu}$ is negative. 
\end{Example}
We now have the following claim, giving a more precise form of Proposition \ref{prop: projcohsmooth}. 
\begin{proposition}
Let $G/\mathbb{Q}_{p}$ be a quasi-split connected reductive group with $P \subset G$ a standard parabolic of $G$. Write $\Bun_{P} \rightarrow \Spd(F)$ for the moduli stack parameterizing $P$-bundles on $X_{S}$ the relative algebraic Fargues-Fontaine curve for $S \in \Perf_{F}$. Then $\Bun_{P}$ is an Artin $v$-stack cohomologically smooth over $\Spd(F)$. The connected components $\Bun_{P}^{(\nu)}$ are of pure $\ell$-dimension equal to $-d_{\nu}$. 
\end{proposition}
\begin{proof}
We apply Theorem \ref{thm: stackyjacobi} with $Z = X_{F}$ and $H = P$, so we have $\mathcal{M}_{[X_{F}/P]} \simeq \Bun_{P}$. Then, as in the proof of Theorem \ref{thm: bungsmooth}, we can see that, given a section $s: X_{F} \rightarrow [X_{F}/P]$ corresponding to a $P$-bundle $\mathcal{G}_{P}$, the pullback $Ls^{*}T^{*}_{[X_{F}/P]}$ is isomorphic to $\mathrm{Lie}(P) \times^{P,Ad} \mathcal{G}_{P}[1]$. From here, one easily verifies the desired cohomological smoothness, as this only has non-trivial cohomology sheaves in degree $-1$. It remains to show the claim on $\ell$-dimension, so suppose that $\mathcal{G}_{P}$ defines a point in $\Bun_{P}^{(\nu)}$. We need to compute the quantity:  
\[ -\deg(\mathrm{Lie}(P) \times^{P,Ad} \mathcal{G}_{P}). \]
To do this, we note that, if $\mathcal{G}_{M} := \mathcal{G}_{P} \times^{P} M$ and $U$ denotes the unipotent radical of $P$ then there exists a short exact sequence of vector bundles 
\[ 0 \ra \Lie(U) \times^{P,\Ad} \mathcal{G}_{P} \ra \Lie(P) \times^{P,\Ad} \mathcal{G}_{P} \ra \Lie(M) \times^{M,\Ad} \mathcal{G}_{M} \ra 0, \]
which gives us an equality 
\[\deg(\mathrm{Lie}(P) \times^{P,\Ad} \mathcal{G}_{P}) = \deg(\Lie(M) \times^{M,\Ad} \mathcal{G}_{M}) + \deg(\Lie(U) \times^{P,\Ad} \mathcal{G}_{P})  = \deg(\Lie(U) \times^{P,\Ad} \mathcal{G}_{P}), \]
where the last equality follows from the fact that $\deg(\Lie(M) \times^{M,\Ad} \mathcal{G}_{M}) = 0$, since $M$ is reductive, as in the proof of Theorem \ref{thm: bungsmooth}. To compute the RHS, we use the tautological equality:
\[ \mathrm{deg}(\mathcal{G}_{P} \times^{P,\Ad} \mathrm{Lie}(U)) = \mathrm{deg}(\mathcal{G}_{P} \times^{P,\Ad} \mathrm{det}(\Lie(U)). \]
However, by assumption, $\mathcal{G}_{M}$ will be an $M$-bundle lying in the same connected component of $\Bun_{M}$ as $\mathcal{F}_{M}^{\nu}$; therefore, $\kappa$ applied to their associated Kottwitz elements will be the same. If we write $\mathcal{F}_{P}^{\nu} := \mathcal{F}_{M}^{\nu} \times^{M} P$ then this implies we have an equality\footnote{More precisely, if $M^{\mathrm{ab}}$ denotes the abelianization of $M$ the assumption implies that the induced $M^{\mathrm{ab}}$-bundles of these two $M$-bundles agree, and from here the equality follows.}
\[ \mathrm{deg}(\mathcal{G}_{P} \times^{P,\Ad} \mathrm{det}(\Lie(U))) = \mathrm{deg}(\mathcal{F}_{P}^{\nu} \times^{P,\Ad} \mathrm{det}(\Lie(U))). \]
However, we note that, by definition of the slope cocharacter the RHS will break up as a sum of roots giving
\[ \langle 2\rho_{U}, \mu(\mathcal{G}^{\nu}_{P}) \rangle \]
where $2\rho_{U} := 2\rho_{G} - 2\rho_{M}$ and $2\rho_{M}$ is the sum of all positive roots of $M$ with respect to the choice of Borel. However, since $\mathcal{G}^{\nu}_{P} \times^{P} M$ is a semistable $M$-bundle by definition, we have that $\langle 2\rho_{M}, \mu(\mathcal{F}_{P}^{\nu}) \rangle = 0$. Thus, this is equal to 
\[ d_{\nu} = \langle 2\rho_{G}, \mu(\mathcal{F}_{P}^{\nu}) \rangle, \]
as desired.
\end{proof}
We will now conclude with a more interesting example, the smoothness of Drinfeld's compactification of the morphism $\mf{p}: \Bun_{B} \ra \Bun_{\GL_{2}}$. This is denoted by $\ol{\mf{p}}: \ol{\Bun}_{B} \ra \Bun_{\GL_{2}}$ and it parametrizes for $S \in \Perf_{F}$ a pair of a line bundle $\mathcal{L}$ and a rank two vector bundle $\mathcal{E}$ on $X_{S}$ together with a map of $\mathcal{O}_{X_{S}}$-modules $\mathcal{L} \hookrightarrow \mathcal{E}$ whose pullback to each geometric point is an injection of coherent sheaves. For a pair of integers $(d_{1},d_{2}) \in \bb{Z}^{2}$, we write $\ol{\Bun}_{B}^{(d_{1},d_{2})} \subset \ol{\Bun}_{B}$ for the open and closed substack determined by fixing the degree of $\mathcal{L}$ to be equal to $d_{1}$ and the degree of $\mathcal{E}/\mathcal{L}$ to be equal to $d_{2}$.
\begin{proposition}
The $v$-stack $\overline{\Bun}_{B}^{(d_{1},d_{2})}$ is an Artin $v$-stack cohomologically smooth over $\Spa(F)$. The open and closed substacks $\overline{\Bun}_{B}^{(d_{1},d_{2})}$ are of pure $\ell$-dimension $d_{2} - d_{1}$.
\end{proposition}
\begin{proof}
We consider $G = \GL_{2}$ and $B = TU$ the decomposition of the Borel into the maximal torus $T$ and the unipotent radical $U$. 

Consider the quotient $G/U$ and write $\ol{G/U}$ for its affine closure. We recall, by \cite{Gro}, that $G/U$ is strongly quasi-affine; in particular, the natural map $G/U \hookrightarrow \ol{G/U}$ is an open immersion. Viewing these as constant schemes over $X_{F}$, we may consider the stack quotient $[G\backslash\ol{G/U}/T] \simeq [\ol{G/U}/(G \times T)] \ra X_{F}$, and we note that $\ol{G/U}$ is actually smooth in this particular case by direct computation (i.e. $G/U \simeq (\bb{A}^{2} \setminus \{0\}) \times \bb{G}_{m}$ via the map recording the , so we are in a situation where we may apply Theorem \ref{thm: stackyjacobi} for $Z = \ol{G/U}$ and $H = G \times T$. 

For $S \in \Perf_{F}$, by \cite{BG}, the moduli space $\mathcal{M}_{[Z/H]}$ parametrizes a pair of a rank $1$-bundle $\mathcal{L}$ on $X_{S}$ and a rank $2$-bundle $\mathcal{E}$ on $X_{S}$ together with a map $\mathcal{O}_{X_{S}}$-modules $\mathcal{L} \ra \mathcal{E}$ (Note that $[G\backslash (G/U)/T] \simeq [X_{F}/B]$ and the open subspace $\Bun_{B} \hookrightarrow \ol{\Bun}_{B}$ corresponds to the locus defined by $G/U \hookrightarrow \ol{G/U}$). We can therefore realize $\ol{\Bun}_{B}$ as an open (cf. Remark \ref{rem: openess}) substack of $\mathcal{M}_{[Z/H]}$ where these define fiberwise injective maps of coherent sheaves after pulling back to each geometric point of $S$. By Theorem \ref{thm: stackyjacobi}, it suffices to show that it lies inside the substack $\mathcal{M}_{[Z/H]}^{\mathrm{sm}} \subset \mathcal{M}_{[Z/H]}$. In order to do this, it suffices to consider a geometric point $x: \Spa(F) \ra \ol{\Bun}_{B}$ corresponding to a fiberwise injective map $\mathcal{L} \hookrightarrow \mathcal{E}$ of $X_{F}$-bundles. If we write $s_{x}$ for the corresponding section then we easily see that $s_{x}^{*}(T^{*}_{[Z/H]/X_{F}})$ has a filtration with graded pieces isomorphic to $\RHom_{\mathcal{O}_{X}}(\mathcal{E}/\mathcal{L},\mathcal{L})[1]$, $\RHom_{\mathcal{O}_{X}}(\mathcal{E}/\mathcal{L},\mathcal{E}/\mathcal{L})[1]$, and $\RHom_{\mathcal{O}_{X}}(\mathcal{L},\mathcal{L})[1]$ (cf. \cite[Lemma~4.4.1]{Lau}), as can be seen by recalling that deformations of a flag $0 \subset \mathcal{L} \subset \mathcal{E}$ to $X_{F}[\epsilon/\epsilon^{2}]$ in the derived category of coherent $X_{F}$-modules can be described as derived self Homs from the flag to itself in the filtered derived category of coherent $X_{F}$-modules, as described in \cite[Tag~05RX]{Stacks}.

To verify that $x$ defines a point in $\mathcal{M}_{[Z/H]}^{\mathrm{sm}}$, it suffices by Proposition \ref{prop: smoothnessconditions} and the long exact sequence attached to the filtration to show that $H^{0}$ of each of these complexes is a direct sum of a vector bundle with positive slopes and a torsion sheaf (recalling that we allow for the possibility that the vector bundle summand is $0$). We verify this using the following lemma.
\begin{lemma}
Let $\mathcal{E}$ and $\mathcal{G}$ be coherent sheaves on $X_{F}$ for $F$ an algebraically closed perfectoid field. If $\mathcal{E}$ is a vector bundle then the object $\RHom_{\mathcal{O}_{X_{F}}}(\mathcal{G},\mathcal{E})[1]$ is representable by a two-term complex $\{\mathcal{E}^{-1} \ra \mathcal{E}^{0}\}$ of vector bundles where $\mathcal{E}_{0}$ is either $0$ or has positive slopes. In particular, in light of  Proposition \ref{prop: smoothnessconditions}, the $\mathcal{O}_{X_{F}}$-module $H^{0}(\RHom_{\mathcal{O}_{X_{F}}}(\mathcal{G},\mathcal{E})[1])$ is a direct sum of a vector bundle with positive slopes and a torsion sheaf. More generally, $H^{0}(\RHom_{\mathcal{O}_{X_{F}}}(\mathcal{G},\mathcal{E})[1])$ is always a direct sum of a vector bundle with positive slopes and a torsion sheaf.
\end{lemma}
\begin{proof}
We first rewrite $\mathcal{G} \simeq \mathcal{G}_{\free} \oplus \mathcal{G}_{\tors}$ and note that it suffices to show the claim separately for $\mathcal{G}_{\free}$ and $\mathcal{G}_{\tors}$. For $\mathcal{G}_{\free}$, we note that we have a quasi-isomorphism 
\[ \RHom_{\mathcal{O}_{X_{F}}}(\mathcal{G}_{\free},\mathcal{E})[1] \simeq \mathcal{G}_{\free}^{\vee} \otimes \mathcal{E}[1] \simeq |\{\mathcal{G}_{\free}^{\vee} \otimes \mathcal{E} \ra 0 \}|, \]
as desired. For $\mathcal{G}_{\tors}$ we can, by taking direct sums, further reduce to the case that $\mathcal{G}_{\tors} \simeq \mathcal{O}_{X_{F},x}/t_{x}^{n}$, where $x$ is a closed point in $X_{F}$ and $t_{x} \in \mathcal{O}_{X_{F},x}$ is a uniformizing element in the local ring at $x$. For $m \in \bb{Z}$ varying, we then have a short exact sequence 
\[ 0 \ra \mathcal{O}(m) \ra \mathcal{O}(n + m) \ra \mathcal{O}_{X_{F},x}/t_{x}^{n} \ra 0, \]
and applying $\RHom_{\mathcal{O}_{X_{F}}}(-,\mathcal{E})$ this yields for us a distinguished triangle 
\[ 
\begin{tikzcd}
& & \RHom_{\mathcal{O}_{X_{F}}}(\mathcal{O}_{X_{F},x}/t_{x}^{n},\mathcal{E})[1] \arrow[dl,"+1"] & \\
& \RHom_{\mathcal{O}_{X_{F}}}(\mathcal{O}(n + m),\mathcal{E}) \arrow[rr] &  & \RHom_{\mathcal{O}_{X_{F}}}(\mathcal{O}(m),\mathcal{E}) \arrow[ul]
\end{tikzcd} 
\] 
where we have implicitly rotated the triangle. Now, as before, we can rewrite the two terms in the bottom of the triangle as $\mathcal{E}(-n - m)$ and $\mathcal{E}(-m)$. In summary, we can find a presentation of $\RHom_{\mathcal{O}_{X_{F}}}(\mathcal{O}_{X_{F},x}/t_{x}^{n},\mathcal{E})[1]$ as $\Cofiber(\mathcal{E}(-n - m) \ra \mathcal{E}(-m))$. Now, by choosing $m$ to be sufficiently negative relative to the slopes of $\mathcal{E}$ we see that we can guarantee that $\mathcal{E}(-m)$ has positive slopes. This gives the desired presentation if $\mathcal{E}$ is a vector bundle and implies the desired claim on the $H^{0}$ in general by looking at the long exact sequence. 
\end{proof}
This claim shows the smoothness, and the claim on the $\ell$-dimension is also easily verified.
\end{proof}

\section{An Approximation Argument}
We now turn our attention to the missing input in the proof of Theorem \ref{thm:main}.
\begin{proposition}{\label{lem:v-local-linearization}}
Let $S \in \Perf$, let $H$ be a linear
algebraic group over $\Qp$, and let $Z\to X_S$ be a smooth
quasi-projective morphism equipped with an $H$-action over $X_S$, the algebraic Fargues-Fontaine curve. Then there exists an affinoid pro-\'etale surjection $S'\to S$ such that if we write $Z'=Z\times_{X_S}X_{S'}$ then, for every closed immersion $H\hookrightarrow\GL_n$, the
contracted product $Z'\times^H\GL_n$ is represented by a smooth quasi-projective scheme over $X_{S'}$.
\end{proposition}
To prove this, we will need to use some approximation arguments on the Fargues-Fontaine curve to degenerate to the case where $X_{S}$ is replaced by normal schemes of finite type over $\bb{Q}_{p}$ in which case the result follows from \cite[Theorem~3]{BrionModels}. For $S = \Spa(R,R^{+}) \in \Perf$ with pseudo-uniformizer $\varpi$, we set $\mathcal{Y}_{S} := \Spa(W(R^{+}),W(R^{+})) \setminus V(p[\varpi])$ and let $B_R=\Gamma(\mathcal{Y}_S,\mathcal O_{\mathcal{Y}_S})$. It is equipped with a Frobenius $\phi$, and we consider the graded ring
\[
  P_{S,*} =\bigoplus_{d\geq0}B_R^{\phi=p^d}, \]
so that $X_{S} := \Proj(P_{S,*})$. We prove the following approximation lemma for this ring. 
\begin{lemma}{\label{lemma: NoetherianApproximation}}
We consider $R^+=\prod_{i\in I}F_i^+$ for some index set $I$, where  the $F_i$ are algebraically closed perfectoid fields in characteristic $p$ and the $F_{i}^{+}$ are open bounded valuation subrings. We let $\varpi = (\varpi_{i})_{i \in I}$ be a pseudo-uniformizer and set $R := R^{+}[\frac{1}{\varpi}]$. This defines a strictly totally disconnected perfectoid space $S := \Spa(R,R^{+})$. We claim that every finitely generated graded $\Qp$-subalgebra
$A\subset P_{S,*}$ is contained in a normal finitely generated graded
$\Qp$-subalgebra of $P_{S,*}$.
\end{lemma}
\begin{proof}
We set $P_{*} := P_{S,*}$ and $P_{i,*} := P_{(F_{i},F_{i}^{+}),*}$. We consider the natural restriction map  $P_{*} \hookrightarrow\prod_i P_{i,*}$. We note that each $P_{i,*}$ is a normal domain: it is the graded ring given by the direct sum of the rings of global sections of the invertible sheaf $\mathcal O_{X_{F_{i}}}(d)$ on the regular integral curve $X_{F_i}$ for $d \geq 0$ \cite[Theorem~II.0.4]{FS}. We normalize  the rank-one valuations on the $F_i$ so that
$|\varpi_i|_i=c$ for a fixed $c\in(0,1)$.

Following \cite[Definition~5.1.1]{KL}, we recall that for all $r \in (0,\infty)$, we consider
\begin{equation}{\label{eqn: defnofintegralRobbaRing}}
  \widetilde{\mathcal R}_{F_i}^{\mathrm{int},r}
  :=
  \{
    \sum_{j\geq 0}p^j[a_j]\in W(F_i)
    |
    \lim_{j\to\infty}p^{-j}|a_j|_i^{\,r}=0\},
\end{equation}
where $[a_j]$ denotes the Teichm\"uller lift of $a_j\in F_i$. We recall for $0<t\leq r$, this ring carries the Gauss norm
\begin{equation}{\label{eqn: GaussNorm}}
  |\sum_{j\geq 0}p^j[a_j]|_t
  =
  \max_{j\geq 0}\bigl\{p^{-j}|a_j|_i^{\,t}\bigr\}.
\end{equation}
We also put
\[
  \widetilde{\mathcal R}_{F_i}^{\mathrm{int}}
  :=
  \bigcup_{r>0}
  \widetilde{\mathcal R}_{F_i}^{\mathrm{int},r}.
\]
We note that we have a natural map 
\begin{equation}{\label{eqn: embedding}}
B_{R} \ra \prod_{i} B_{F_{i}} 
\end{equation}
induced by restriction. We will want to use this map to reduce to the case that $(R,R^{+}) = (F,F^{+})$ in which case the claim follows by taking normalizations since $P_{S,*}$ will be a normal domain. To do this, we have the following claim.
\begin{lemma}{\label{lemma: Emebeddingintoalgebraicallyclosedfields}}
The map (\ref{eqn: embedding}) gives an isomorphism
\begin{equation}\label{eq:bounded-families}
 B_R \xrightarrow{\simeq}
 \{(b_i)\in\prod_i B_{F_i} |
       \sup_i|b_i|_t<\infty\text{ for every }t \in (0,\infty)\},
\end{equation}
where $|\cdot|_t$ is the Gauss norm with parameter $t$ defined with respect to the Teichm\"uller expansion of elements in $B_{F_{i}}$ and the rank one valuation on $F_{i}$ on bounded elements (cf. \ref{eqn: GaussNorm}) and then extended to the Fr\'echet completion using continuity. 
\end{lemma}
\begin{proof}
Restriction is isometric for the corresponding supremum seminorm so (\ref{eqn: embedding}) is injective and its image is contained in the right-hand side of (\ref{eq:bounded-families}) (Recall that $B_{R}$ and $B_{F_{i}}$ are Hausdorff since they may be described as Frechet completions with respect to the Gauss norm). 

For the converse, fix a family $(b_i)$ in the right-hand side of (\ref{eq:bounded-families}) and a compact interval $[\alpha,\beta]\subset(0,\infty)$. Choose $0< \alpha_0< \alpha \leq \beta < \beta_0$ and put $M_t :=\sup_i|b_i|_t$, which is finite by definition of the RHS  of (\ref{eq:bounded-families}).

We apply \cite[Lemma~5.2.8]{KL} at the single radius $\alpha_0$.
For each integer $n$, it gives decompositions $b_i=y_i+z_i$ with
$y_i\in p^n\widetilde{\mathcal R}_{F_i}^{\mathrm{int},\alpha_0}$,
with $z_i$ extending to all radii $t\geq \alpha_0$, and with
\[
  |z_i|_t\leq
  p^{(1-n)(1-t/\alpha_0)}M_{\alpha_0}^{t/\alpha_0} \text{ , } (t\geq \alpha_0).
\]
As $n\to-\infty$, this bound tends uniformly to zero on $[\alpha,\beta]$.
Moreover, $y_i=b_i-z_i$ extends beyond $\beta_0$.
By \cite[Lemma~5.2.6]{KL}, it belongs to
$p^n\widetilde{\mathcal R}_{F_i}^{\mathrm{int},\beta_0}$.
For fixed $n$ there is a constant $K$, independent of $i$, such that
$|y_i|_{\beta_0}\leq K$.

The integral elements $y_i$ have Witt expansions
\[
  y_i=\sum_{j\geq n}p^j[u_{ij}],
\]
by definition (\ref{eqn: defnofintegralRobbaRing}) and the ultra-metric inequality implies  $p^{-j}|u_{ij}|^{\beta_0}\leq K$. In summary, we have shown for fixed $j$ that the element $p^{-j}|u_{ij}|^{\beta_0}\leq K$ for all $i$. Therefore, there exists a single $M > 0$ such that $\varpi_{i}^{M}u_{ij} \in F_{i}^{+}$, which implies that $u_{j} := (u_{ij})_{i \in I} \in R$ for fixed $j$ and all $i$. Furthermore, for $t\in[\alpha,\beta] \subset (\alpha_{0},\beta_{0})$, we have that
\[
  |p^j[u_{ij}]|_t
  \leq K^{t/\beta_0}p^{-j(1-t/\beta_0)}.
\]
Hence, the tails with $j>N$ tend uniformly to zero on $[\alpha,\beta]$,
uniformly in $i$.
It follows that the finite sums $\sum_{j=n}^N p^j[u_j]$, which belong to
$B_R$ as explained above, approximate $(b_i)_{i \in I}$ uniformly on $[\alpha,\beta]$.
Exhausting $(0,\infty)$ by compact intervals and using the completeness
of $B_R$ proves the isomorphism \eqref{eq:bounded-families}.
\end{proof}
Suppose that homogeneous elements $b_i\in P_{i,d}$
satisfy a polynomial equation
\begin{equation}\label{eq:integrality}
 b_i^m+a_{1,i}b_i^{m-1}+\cdots+a_{m,i}=0,
\end{equation}
with coefficients $a_j\in P_{jd}$. The multiplicativity and triangle inequality for the Gauss norms gives
\begin{equation}{\label{eqn: inequality}}
  |b_i|_t\leq
  \max_{1\leq j\leq m}|a_{j,i}|_t^{1/j}.
\end{equation}
Thus, the isomorphism (\ref{eq:bounded-families}) assembles the $b_i$ to an element
$b\in B_R$ (using (\ref{eqn: inequality})). Since $\phi(b_i)=p^db_i$ for every $i$,
injectivity of the map (\ref{eqn: embedding}) gives $b\in P_d$. The same argument applies after restriction to any subset of $I$.

For every subset $J\subset I$ of the index set, its characteristic function defines an idempotent $e_J\in P_0$. We claim that every finitely generated graded $\Qp$-subalgebra $A\subset P_{*}$ is contained in a normal finitely generated graded
$\Qp$-subalgebra of $P_{*}$. We prove this by induction on $\dim(A)$, simultaneously for all
subsets of $I$ (where we note if $\dim(A) = 0$ then it is a product of fields and therefore is already normal).

First suppose that $A$ is a domain.
Its normalization $\overline A$ in $\Frac(A)$ is a finite,
nonnegatively graded $A$-algebra.
Choose a non-zero homogeneous element $c\in A$ (e.g. a non-zero element of the conductor), so that
$\overline A\subset A[1/c]$, and set
\[
  J=\{i\in I:c_i\neq0\},\qquad K=I\setminus J.
\]
The map $A\to\prod_{i\in J}P_{i,*}$ is injective. Indeed, if $a$ vanishes on
$J$, then $ac$ vanishes on every index, whence $ac=0$ in $A$ and
$a=0$.
For $i\in J$, the map $A[1/c]\to\Frac(P_{i,*})$ carries
$\overline A$ into $P_{i,*}$, since $P_{i,*}$ is normal.
Choose homogeneous integral generators of $\overline A$ over $A$.
Each satisfies a homogeneous monic equation of the form
\eqref{eq:integrality}; the preceding argument therefore assembles
its images to an element of $e_JP_{*}$.
We obtain an injective (note that if $x \mapsto 0$ then $cx \in A$ vanishes on $J$ so that $x = 0$) graded homomorphism
\[
  \overline A\hookrightarrow e_JP_{*}
\]
whose image contains $e_JA$. 

The image $A_K$ of $A$ in $e_KP_{*}$ is a reduced finitely generated
graded algebra, and by definition of $K$ we have that it factors through $A/(c)$. If $K\neq\varnothing$ then $\dim(A_K) \leq \dim(A/(c)) < \dim(A)$ and, by induction,
$A_K$ is contained in a normal finitely generated graded subalgebra
of $e_KP_{*}$.
Taking the product with the image of $\overline A$ in $e_JP_{*}$
proves the claim when $A$ is a domain.

For general $A$, let $\mathfrak p_1,\ldots,\mathfrak p_r$ be its
minimal primes, which are homogeneous.
The kernel of $A\to P_{i,*}$ contains at least one $\mathfrak p_j$ by virtue of the fact that $P_{i,*}$ is a domain as noted above.
Choose a finite partition $I=\coprod_j I_j$ so that the maps on
$I_j$ factor through $A/\mathfrak p_j$.
The image $A_j$ in $e_{I_j}P_{*}$ is either isomorphic to the domain
$A/\mathfrak p_j$, in which case the argument in the case of a domain applies,
or is a quotient of strictly smaller dimension, in which case
induction applies.
The product of the resulting normal subalgebras contains $A$. This proves the claim.
\end{proof}
We now deduce the following key Noetherian approximation result for the algebraic Fargues-Fontaine curve. 
\begin{corollary}{\label{cor: geometricNoetherianApproximation}}
Let $S = \Spa(R,R^{+})$ be as in Lemma \ref{lemma: NoetherianApproximation} and $P_{*} := P_{S,*}$ as before. Then we have that
\begin{equation}{\label{eqn: GradedRingApproximation}}
  P_{*} =\varinjlim_\lambda A_{\lambda,*}
\end{equation}
is the colimit over its normal finitely generated graded $\Qp$-subalgebras. Similarly, there exists a cofinal system of $\lambda$ and $T_{\lambda}$ a normal quasi-projective scheme of finite type over $\bb{Q}_{p}$, such that 
\begin{equation}{\label{eqn: NoetherianRingApproximation}}
 X_{S} \simeq \varprojlim_{\lambda} T_{\lambda}. 
\end{equation}
and the transition morphisms are affine.
\end{corollary}
\begin{proof}
The equality (\ref{eqn: GradedRingApproximation}) follows immediately from Lemma \ref{lemma: NoetherianApproximation}. To see the second claim, we choose finitely many positive-degree homogeneous elements $t_1,\ldots,t_r\in P_{*}$ such that the $D_+(t_j)$ cover $X_S$. Such a finite collection exists by the quasi-compactness in \cite[Proposition II.2.7]{FS}.
Restrict to the cofinal system of $A_\lambda$ containing these
elements, and put
\[
  T_\lambda=\bigcup_{j=1}^rD_+(t_j)
  \subset\Proj(A_{\lambda,*}).
\]
Each $T_\lambda$ is a normal quasi-projective scheme of finite type
over $\Qp$.
The transition morphisms are affine: the inverse image of each
$D_+(t_j)$ is the corresponding affine standard open.
On these affine opens, taking degree-zero localizations commutes
with the filtered colimit. Hence, we have an isomorphism
\[
  X_S\simeq\varprojlim_\lambda T_\lambda.
\]

\end{proof}
Corollary \ref{cor: geometricNoetherianApproximation} would be sufficient to prove Proposition \ref{lem:v-local-linearization} if we were allowed to take $S' \ra S$ to be a $v$-cover. Unfortunately, for the application, we require it to be quasi-pro-\'etale. In order to fix this, we show the following lemma.
\begin{lemma}\label{lem:finite-sections}
Let $S \in \Perf$, and let
$f: \mc{F} \to X_S$ be finite and finitely presented. Then the 
$v$-sheaf $\mathcal{M}_{\mc{F}}$ parameterizing sections over $X_{S}$ is separated and quasi-pro-\'etale over $S$.
\end{lemma}

\begin{proof}
We write $\mc{F} = \underline{\Spec}_{X_{S}}(\mc{A})$ for $\mc{A} := f_{*}\mathcal{O}_{\mc{F}}$ a finitely presented $\mc{O}_{X_{S}}$-module. Since $S$ is affinoid perfectoid, the scheme $X_S$ is quasi-compact and $\mathcal O_{X_S}(1)$
is ample by the discussion following \cite[Proposition~II.2.7]{FS}. Therefore, we have a surjection
\[ \mc{O}_{X_{S}}^{\oplus N}(-m) \ra \mc{A}, \]
for some $m\geq0$ and $N$ sufficiently large. This gives a closed embedding of $\mc{F}$ into the total space of $\mathcal{O}_{X_S}(m)^{\oplus N}$. Passing to space of sections, these cut out $\mathcal{M}_{\mc{F}}$ as a closed
subdiamond of
$\mathcal{H}^{0}(\mathcal O_{X_S}(m)^{\oplus N})$. After pullback to a perfectoid test space, their vanishing on
the relative curve is a Zariski closed condition by
\cite[Lemma~IV.4.23]{FS}. Thus,  $\mc{M}_{\mc{F}} \to S$ is separated and representable in locally spatial diamonds.

By \cite[Proposition~13.6]{Ecod}, we may check the claim after pullback to a point
$\Spa(F,\mathcal O_F)\to S$. We write $\mc{F}_{F}$ for the pullback of $\mc{F}$ along the induced map $X_{F} \ra X_{S}$. Choose $d>m$ distinct degree-one Cartier divisors
$D_1,\ldots,D_d$ on $X_F$, and put $D :=\sum_jD_j$.
Evaluation of a section on the Cartier divisors defines an injective map
\[
  \mathcal M_{\mc{F}_{F}}\longrightarrow
  \prod_{j=1}^d(\mc{F}_{F}|_{D_j})^{\Diamond}
\]
of $v$-sheaves over $\Spa(F,\mathcal{O}_{F})$. Indeed, two sections with the same images under the evaluation map differ by a global section of
$\mathcal O_{X_{F}}(m)(-D)^{\oplus N}$.
This bundle has negative slopes, so its global sections vanish.

Each Cartier divisor $D_j$ has residue field given by an algebraically closed untilt of $F$. By the assumed finiteness of $\mathcal{F}_{F}$ over $X_{F}$, each factor on the right is a finite
constant diamond over $\Spa(F,\mathcal{O}_{F})$.
Every $v$-subsheaf of a finite constant diamond over $\Spa(F,\mathcal{O}_{F})$ is a disjoint union of copies of $\Spa(F,\mathcal{O}_{F})$: a nonempty such subsheaf
contains a geometric point, whose map to $\Spa(F,\mathcal{O}_{F})$ is a $v$-cover. Hence $\mathcal{M}_{\mc{F}_{F}}$ is itself finite constant over $\Spa(F,\mathcal{O}_{F})$, and in particular pro-\'etale.
\end{proof}
With this in hand, we may now finally prove Proposition \ref{lem:v-local-linearization}. 
\begin{proof}{(Proof of Proposition \ref{lem:v-local-linearization})}
After an affinoid pro-\'etale cover, we may assume that $S$
is strictly totally disconnected, by
\cite[Lemma~7.18]{Ecod}.
Writing the graded ring defining $X_S$ as the filtered union
of its finitely generated graded subalgebras over $\bb{Q}_{p}$, taking $\mathrm{Proj}(-)$ gives finite-type quasi-projective approximations of $X_{S}$, 
\[ X_{S} \simeq \varprojlim_{\lambda} T_{\lambda}\]
as in Corollary \ref{cor: geometricNoetherianApproximation} and its proof, where unlike this corollary we do not require the $T_{\lambda}$ to be normal. By standard Noetherian approximation arguments \cite[Tag 01ZM and Section~32.8, in particular Tag 0C0C]{Stacks}, and choosing $\lambda$ to be sufficiently large the morphism $Z \ra X_{S}$ descends to some $Z_{0} \ra T_{0}$, where $X_{S} \ra T_{0}$ is the projection and $Z_{0} \ra T_{0}$  is smooth and quasi-projective and carries an $H$-action such that
\begin{equation}{\label{eqn: KeyBaseChangeRelation}}
 Z\simeq Z_0\times_{T_0}X_S 
\end{equation}
in a way that is compatible with the $H$-action in $Z$.

There is a finite morphism $T_1\to T_0$, with $T_1$ normal,
such that every morphism from a normal finite-type scheme $X$ to
$T_0$ lifts to $T_1$.
To construct $T_{1}$, first consider the normalization of the irreducible components
of $(T_0)_{\mathrm{red}}$. Consider the closed complements of the locus inside $(T_{0})_{\mathrm{red}}$ where this normalization map is an isomorphism, and take $T_{1}$ to be the disjoint union of the spaces obtained by iterating this process. Normalization is finite by
\cite[Tag~0BXR]{Stacks}, and the process terminates by Noetherian induction giving rise to a well-defined $T_{1}$.
For a normal finite type scheme $X \ra T_{0}$ for constructing a lift to $T_{1}$ we may assume $X$ is connected hence integral, if its generic point maps into one of the loci where the first normalization of the reduced is an isomorphism, the resulting rational map extends to $T_{1}$. Namely, the closure of its graph is finite birational over the normal $X$ and hence isomorphic to it. Otherwise, the morphism factors through one of the closed subsets where the first normalization failed to be an isomorphism, and then we repeat the argument with respect to the next step of the procedure used in defining $T_{1}$.

Set $\mc{F} =X_S\times_{T_0}T_1$. By Lemma~\ref{lem:finite-sections}, $\mc{M}_{F} \ra S$ is quasi-pro-\'etale. It is also $v$-surjective.
Indeed, choose a $v$-cover $S' \ra S$ by a product of
algebraically closed points. Corollary \ref{cor: geometricNoetherianApproximation} writes
$X_{S'}$ as an inverse limit of normal finite-type schemes
with affine transition morphisms. The map $X_{S'}\to T_0$ therefore factors through one of these normal schemes, which lifts to $T_1$ by its construction. This gives a section of $\mc{M}_{\mc{F}}$ over $S'$.

Since $S$ is strictly totally disconnected, $\mc{M}_{\mc{F}}$ is a
perfectoid space pro-\'etale over $S$, by \cite[Definition~10.1(i)]{Ecod}. Since $S^{'}$ is affinoid, the image $\sigma(|S^{'}|)\subset |\mc{M}_{\mc{F}}|$ is quasi-compact, and so it is covered by finitely many affinoid opens $U_1,\ldots,U_k\subset \mc{M}_{\mc{F}}$. As $S^{'} \ra S$ is surjective, the images of the $U_j$ in $|S|$ cover $|S|$. Each $U_j\to S$ is affinoid pro-\'etale by \cite[Lemma~7.19]{Ecod}, so $S^{''}:=\bigsqcup_{j=1}^k U_j\to S$ is an affinoid pro-\'etale cover. The tautological map $S'\to \mc{M}_{\mc{F}}$ then gives a lift $X_{S''}\to T_1$.

Now
\begin{equation}{\label{eqn: KeyBaseChangeRelationshipII}}
Z_1=Z_0\times_{T_0}T_1
\end{equation}
is normal and quasi-projective
over $\Qp$. Therefore, \cite[Theorem~3]{BrionModels} tells us that $\GL_{n} \times^H Z_1$ is normal quasi-projective over $\bb{Q}_{p}$. Moreover, since the action of $H$ on $Z_{0}$ is defined over $T_{0}$, we have a natural map $\GL_{n} \times^H Z_1 \ra T_{1}$ which will be quasi-projective since $T_{1}$ is separated. Moreover, there is a natural $H$-torsor
\[ \GL_{n} \times Z_{1} \ra \GL_{n} \times^{H} Z_{1} \]
which is smooth surjective since $H$ is smooth over $\bb{Q}_{p}$ and the source is smooth over $T_{1}$ by assumption. Therefore, using \cite[Tag~02K5]{Stacks}, $\GL_{n} \times^{H} Z_{1} \ra T_{1}$ is smooth quasi-projective. Now, since contracted product commutes with base-change, we have that
\[ \GL_{n} \times^{H} Z \times_{X_{S}} X_{S''} \simeq \GL_{n} \times^{H} Z_{1} \times_{T_{1}} X_{S''} \]
by combining with (\ref{eqn: KeyBaseChangeRelation}) and (\ref{eqn: KeyBaseChangeRelationshipII}). This shows the desired claim since $S' \ra S$ is an affinoid pro-\'etale cover.
\end{proof}

\printbibliography 
\end{document}